\documentclass[11pt,a4paper,reqno]{amsart}
\usepackage[english]{babel}
\usepackage[T1]{fontenc}
\usepackage{verbatim}
\usepackage{palatino}
\usepackage{amsmath}
\usepackage{amssymb}
\usepackage{amsthm}
\usepackage{amsfonts}
\usepackage{graphicx}
\usepackage{color}
\usepackage{mathtools}
\usepackage{esint}

\usepackage[colorlinks = true, citecolor = black]{hyperref}
\title[Riesz transform and vertical oscillation]{Riesz transform and vertical oscillation \\ in the Heisenberg group}
\author[K. F\"assler and T. Orponen]{Katrin F\"assler and Tuomas Orponen}
\address{Department of Mathematics\\ University of Fribourg \\ Chemin du Mus\'{e}e 23,
CH-1700 Fribourg, Switzerland}
\address{Department of Mathematics and Statistics, University of Helsinki, P.O. Box 68 (Pietari Kalmin katu 5), FI-00014 University of Helsinki}
\email{katrin.faessler@unifr.ch}
\email{tuomas.orponen@helsinki.fi}
\date{\today}
\subjclass[2010]{42B20 (Primary) 31C05, 35R03, 32U30, 28A78 (Secondary) }
\thanks{K.F.\ was supported by the Swiss National Science Foundation through project 161299 \emph{Intrinsic rectifiability and mapping theory on the Heisenberg group}. T.O. was supported by the Finnish Academy through the project \emph{Quantitative rectifiability in Euclidean and non-Euclidean spaces}, grant nos. 309365 and 314172.}
\keywords{Singular integrals, Riesz transform, intrinsic Lipschitz graphs, Heisenberg group}

\newcommand{\R}{\mathbb{R}}
\newcommand{\W}{\mathbb{W}}
\newcommand{\He}{\mathbb{H}}
\newcommand{\N}{\mathbb{N}}

\newcommand{\V}{\mathbb{V}}

\newcommand{\calH}{\mathcal{H}}

\newcommand{\calB}{\mathcal{B}}

\newcommand{\calS}{\mathcal{S}}

\newcommand{\calR}{\mathcal{R}}

\newcommand{\spt}{\operatorname{spt}}
\newcommand{\Hd}{\dim_{\mathrm{H}}}

\newcommand{\diam}{\operatorname{diam}}
\newcommand{\card}{\operatorname{card}}
\newcommand{\dist}{\operatorname{dist}}

\newcommand{\Div}{\mathrm{div}_{\He}}

\newcommand{\osc}{\operatorname{osc}}
\newcommand{\w}{\mathrm{v}}

\def\Barint_#1{\mathchoice
          {\mathop{\vrule width 6pt height 3 pt depth -2.5pt
                  \kern -8pt \intop}\nolimits_{#1}}%
          {\mathop{\vrule width 5pt height 3 pt depth -2.6pt
                  \kern -6pt \intop}\nolimits_{#1}}%
          {\mathop{\vrule width 5pt height 3 pt depth -2.6pt
                  \kern -6pt \intop}\nolimits_{#1}}%
          {\mathop{\vrule width 5pt height 3 pt depth -2.6pt
                  \kern -6pt \intop}\nolimits_{#1}}}

\numberwithin{equation}{section}

\theoremstyle{plain}
\newtheorem{thm}[equation]{Theorem}

\newtheorem{lemma}[equation]{Lemma}

\newtheorem{ex}[equation]{Example}
\newtheorem{cor}[equation]{Corollary}
\newtheorem{proposition}[equation]{Proposition}
\newtheorem{question}{Question}

\theoremstyle{definition}

\newtheorem{definition}[equation]{Definition}

\theoremstyle{remark}
\newtheorem{remark}[equation]{Remark}

\begin{document}

\begin{abstract} We study the $L^{2}$-boundedness of the $3$-dimensional (Heisenberg) Riesz transform on intrinsic Lipschitz graphs in the first Heisenberg group $\He$. Inspired by the notion of vertical perimeter, recently defined and studied by Lafforgue, Naor, and Young, we first introduce new scale and translation invariant coefficients $\operatorname{osc}_{\Omega}(B(q,r))$. These coefficients quantify the vertical oscillation of a domain $\Omega \subset \He$ around a point $q \in \partial \Omega$, at scale $r > 0$. We then proceed to show that if $\Omega$ is a domain bounded by an intrinsic Lipschitz graph $\Gamma$, and
\begin{displaymath} \int_{0}^{\infty} \operatorname{osc}_{\Omega}(B(q,r)) \, \frac{dr}{r} \leq C < \infty, \qquad q \in \Gamma, \end{displaymath}
then the Riesz transform is $L^{2}$-bounded on $\Gamma$. As an application, we deduce the boundedness of the Riesz transform whenever the intrinsic Lipschitz parametrisation of $\Gamma$ is an $\epsilon$ better than $\tfrac{1}{2}$-H\"older continuous in the vertical direction.

We also study the connections between the vertical oscillation coefficients, the vertical perimeter, and the natural Heisenberg analogues of the $\beta$-numbers of Jones, David, and Semmes. Notably, we show that the $L^{p}$-vertical perimeter of an intrinsic Lipschitz domain $\Omega$ is controlled from above by the $p^{th}$ powers of the $L^{1}$-based $\beta$-numbers of $\partial \Omega$.
\end{abstract}

\maketitle

\tableofcontents

\section{Introduction}

\subsection{A Euclidean introduction to the Heisenberg Riesz transform} A fundamental singular integral operator (SIO) in $\R^{d}$ is the \emph{$(d - 1)$-dimensional Riesz transform}, formally defined by the convolution
\begin{displaymath} R_{d - 1}\nu(x) = \nu \ast \frac{x}{|x|^{d}}. \end{displaymath}
Here $x/|x|^{d}$ is the \emph{$(d - 1)$-dimensional Riesz kernel} which is, up to a constant, the gradient of the fundamental solution of the Laplacian. Through this connection to the Laplace equation, the operator $R_{d - 1}$ has many applications to problems concerning analytic and harmonic functions. For instance, whenever $R_{d - 1}$ is bounded on $L^{2}(\mu)$ for a $(d - 1)$-regular measure $\mu$, then the support of $\mu$ is non-removable for Lipschitz harmonic functions (or bounded analytic functions in the plane); see the book \cite{tolsabook} of Tolsa for an in-depth introduction to this topic and many more references.

A second application of the SIO $R_{d - 1}$ is the \emph{method of layer potentials} employed to solve the Dirichlet problem
\begin{equation}\label{dirichlet} \begin{cases} \bigtriangleup u(x) = 0, & x \in \Omega, \\ u|_{\partial \Omega} = g, \end{cases} \end{equation}
on domains $\Omega \subset \R^{d}$ with Lipschitz boundaries, and with, say, $g \in L^{2}(\calH^{d - 1}|_{\partial \Omega})$.  As the name suggests, a key component in the method of layer potentials is the study of the \emph{boundary layer potential}
\begin{displaymath} D\nu(x) = \mathrm{p.v.} \frac{1}{\omega_{d}} \int_{\partial \Omega} \frac{(y - x) \cdot n_{\partial \Omega}(y)}{|y - x|^{d}} \, d\nu(y). \end{displaymath}
The boundedness of the operator $D$ on $L^{2}(\calH^{d - 1}|_{\partial \Omega})$ can be derived from the boundedness of $R_{d - 1}$, see \cite{MR501367,MR769382}.

By now, the $L^{2}$-boundedness properties of the operator $R_{d - 1}$ are well-understood. According to a result of David and Semmes \cite{DS1}, generalising earlier works of Calder\'on \cite{Calderon} and Coifman, McIntosh, and Meyer \cite{CMM}, $R_{d - 1}$ is bounded on $L^{2}(\calH^{d - 1}|_{S})$ whenever $S \subset \R^{d}$ is \emph{uniformly $(d - 1)$-rectifiable}.  More recently, Nazarov, Tolsa, and Volberg \cite{ntov} proved a converse: if $S \subset \R^{d}$ is $(d - 1)$-regular, then the uniform rectifiability of $S$ is necessary for the boundedness of $R_{d - 1}$ on $L^{2}(\calH^{d - 1}|_{S})$. These results have been used to show that a compact $(d - 1)$-set is removable for Lipschitz harmonic functions if and only if it is purely $(d - 1)$-unrectifiable \cite{MR1372240,ntov2} and that the Dirichlet problem \eqref{dirichlet} is solvable in Lipschitz domains with $L^{2}$-boundary values \cite{MR769382}.

The work in the current paper is motivated by aspirations to extend parts of the theory above to the case of a basic hypoelliptic and non-elliptic operator, the \emph{sub-Laplacian} (also known as the \emph{Kohn Laplacian})
\begin{displaymath} \bigtriangleup_{\He} = X^{2} + Y^{2} \end{displaymath}
in $\R^{3}$. Here $X$ and $Y$ are the vector fields
\begin{equation}\label{eq:LeftInvVfd} X = \partial_{x} - \frac{y}{2}\partial_{t} \quad \text{and} \quad Y = \partial_{y} + \frac{x}{2}\partial_{t}. \end{equation}
A first step is to understand the $L^{2}$-boundedness of an associated "Riesz transform" operator, which we will soon define.

Whereas the operators $X,Y,\bigtriangleup_{\He}$ do not interact particularly nicely with Euclidean translations, they do commute with the following "left translations" $\tau_{p} \colon \R^{3} \to \R^{3}$,
\begin{displaymath} \tau_{p}(q) := (x + x', y + y', t + t' + \tfrac{1}{2}(xy' - x'y)), \end{displaymath}
where $p = (x,y,t) \in \R^{3}$ and $q = (x',y',t') \in \R^{3}$. This suggests that it is natural to study questions about $\bigtriangleup_{\He}$ in the setting of the first \emph{Heisenberg group} $\He = (\R^{3},\cdot)$, where the group law "$\cdot$" is \textbf{defined} so that $X$ and $Y$ are (left) invariant:
\begin{displaymath} p \cdot q := \tau_{p}(q). \end{displaymath}
It was shown by Folland \cite{MR0494315} that the operator $\bigtriangleup_{\He}$ has a fundamental solution $G \colon \R^{3} \setminus \{0\} \to \R$, whose formula is given by
\begin{displaymath} G(p) = \frac{c}{((x^{2} + y^{2})^{2} + 16t^{2})^{1/2}} =: \frac{c}{\|p\|_{Kor}^{2}}, \qquad p = (x,y,t) \in \He \setminus \{0\}. \end{displaymath}
Here $c > 0$ is a constant, and $\|p\|_{Kor} := ((x^{2} + y^{2})^2 + 16 t^{2})^{1/4}$. This quantity is known as the \emph{Kor\'anyi norm} of the point $p \in \He$, and it induces a metric $d_{Kor}$ on $\He$ via the relation
\begin{equation}\label{eq:Kor} d_{Kor}(p,q) = \|q^{-1} \cdot p\|_{Kor}. \end{equation}
The distance $d_{Kor}$ is invariant under the left translations, that is, $d_{Kor}(p \cdot q_{1},p \cdot q_{2}) = d_{K}(q_{1},q_{2})$ for all $p,q_{1},q_{2} \in \He$.

In analogy with the $(d - 1)$-dimensional Riesz transform discussed above, one may now consider the SIO $R$ formally defined by
\begin{displaymath} R\nu(p) := \nu \ast \nabla_{\He} G(p). \end{displaymath}
Here $\nabla_{\He}$ stands for the \emph{horizontal gradient} $\nabla_{\He}G = (XG,YG)$, and the convolution should be understood in the Heisenberg sense:
\begin{displaymath} f \ast g(p) = \int f(q)g(q^{-1} \cdot p) \, dq. \end{displaymath}
The main open question is the following:
\begin{question}\label{mainQ} For which locally finite Borel measures $\mu$ on $\He$ (equivalently $\R^{3}$) is the operator $R$ bounded on $L^{2}(\mu)$? \end{question}
Here, the boundedness on $L^{2}(\mu)$ is defined in the standard way via $\epsilon$-truncations; see Section \ref{s:Riesz} for the precise definition.

\subsection{Previous work} To the best of our knowledge, the Heisenberg Riesz transform $R$ was first mentioned in the paper \cite{CM} of Chousionis and Mattila, where the following removability question was raised and studied: which subsets of $\He$ (more generally, of Heisenberg groups of arbitrary dimensions) are removable for Lipschitz harmonic functions? The notions of 'Lipschitz' and 'harmonic' should be interpreted in the Heisenberg sense: we call a function $u \colon \He \to \R$ \emph{harmonic} if it solves the sub-Laplace equation $\bigtriangleup_{\He} u = 0$. A function $f \colon \He \to \R$ is \emph{Lipschitz} if $|f(p) - f(q)| \leq Ld_{Kor}(p,q)$ for some $L \geq 1$ and all $p,q \in \He$.

It was shown in \cite[Theorem 3.13]{CM} that the critical exponent for the removability problem in $\He$ is $3$ (keeping in mind that $\Hd (\He,d_{Kor}) = 4$). More precisely, sets with vanishing $3$-dimensional measure are removable, while sets of Hausdorff dimension exceeding $3$ are not. In \cite[Section 5]{CM}, the authors formulate (essentially) Question \ref{mainQ} and suggest its connection to the removability problem.

The connection was formalised by Chousionis and the authors in \cite{CFO2}:
\begin{thm}[Theorem 1.2 in \cite{CFO2}] If $\mu$ is a $3$-regular measure on $\He$ (see \eqref{3regular} below), and $R$ is bounded on $L^{2}(\mu)$, then $\spt \mu$ is non-removable for Lipschitz harmonic functions in $\He$. \end{thm}

In \cite{CFO2}, we also proved the first non-trivial results on the $L^{2}$-boundedness of $R$ (and a class of other SIOs). To discuss these results, and also the ones in the present paper, we need the concept of \emph{intrinsic Lipschitz functions and graphs}. A \emph{vertical subgroup} $\W \subset \He$ is, from a geometric point of view, any $2$-dimensional subspace of $\R^{3}$ containing the $t$-axis. The \emph{complementary horizontal subgroup of $\W$} is the line $\V = \W^{\perp}$ in the $xy$-plane.

 We give the definition of \emph{intrinsic Lipschitz functions} $\phi:\W \to \V$ and the associated \emph{intrinsic Lipschitz graphs} $\Gamma_{\phi}\subset \He$ in Section \ref{ss:intrLip}. These objects were introduced in 2006 by Franchi, Serapioni and Serra Cassano \cite{FSS}, and they appear to be fundamental building blocks in the theory of "high-dimensional" rectifiability in the Heisenberg group, see for example \cite{MR2789472,CFO}. In particular, intrinsic Lipschitz graphs $\Gamma \subset \He$ are closed $3$-regular sets, which means that the measure $\mu = \calH^{3}|_{\Gamma}$ satisfies
\begin{equation}\label{3regular} \mu(B(p,r)) \sim r^{3}, \qquad p \in \spt \mu, \: 0 < r \leq \diam(\spt \mu). \end{equation}
In another paper of Franchi, Serapioni, and Serra Cassano \cite{FSSC2}, a Rademacher-type theorem was established for intrinsic Lipschitz functions: without delving into detail, we just mention that if $\phi \colon \W \to \V$ is intrinsic Lipschitz, then for Lebesgue almost every $w \in \W$ there exists an \emph{intrinsic gradient} for $\phi$, denoted by $\nabla^{\phi}\phi(w)$.

Recall that in $\R^{d}$, Calder\'on \cite{Calderon} and Coifman-McIntosh-Meyer \cite{CMM} proved that $R_{d - 1}$ is bounded on $L^{2}(\calH^{d - 1}|_{\Gamma})$ if $\Gamma \subset \R^{d}$ is a Lipschitz graph. In analogy, one can ask:
\begin{question} Assume that $\Gamma \subset \He$ is an intrinsic Lipschitz graph. Is $R$ bounded on $L^{2}(\calH^{3}|_{\Gamma})$? \end{question}
We are not convinced enough to upgrade the question into a conjecture. In \cite{CFO2}, we obtained a positive answer under a extra regularity:
\begin{thm}[Theorem 1.1 in \cite{CFO2}] Assume that $\alpha > 0$, and $\phi \in C^{1,\alpha}(\W)$ has compact support. Then $R$ is bounded on $L^{2}(\calH^{3}|_{\Gamma_{\phi}})$. \end{thm}
The assumption $\phi \in C^{1,\alpha}(\W)$ means that the intrinsic gradient of $\phi$ exists everywhere and satisfies an intrinsic version of $\alpha$-H\"older regularity (which is weaker than Euclidean $\alpha$-H\"older regularity). The assumption implies, see \cite[Proposition 4.1]{CFO2}, that the affine approximation of $\Gamma_{\phi}$ at $p \in \Gamma$ improves at a geometric rate as one zooms into $p$.

\subsection{New results} A novelty of the current paper is to prove the $L^{2}$-boundedness of $R$ in some scenarios where there is no "pointwise decay" for the quality of affine approximation of $\Gamma$. As a basic example, Theorem \ref{main} below applies to graphs of the form
\begin{displaymath} \Gamma = \Gamma_{\R^{2}} \times \R \subset \He, \end{displaymath}
where $\Gamma_{\R^{2}}$ is a (Euclidean) Lipschitz graph in $\R^{2}$. It turns out that a key feature of these graphs is the following. The two complementary domains $\Omega_{1},\Omega_{2} \subset \He \setminus \Gamma$ have zero "vertical oscillation": for $j \in \{1,2\}$, every vertical line $\ell \subset \He$ satisfies
\begin{equation}\label{alternative} \ell \subset \Omega_{j} \quad \text{or} \quad \ell \cap \Omega_{j} = \emptyset. \end{equation}
The condition \eqref{alternative} is qualitative, not to mention exceedingly restrictive, so we looked for a way to quantify and relax it. For these purposes, we introduce the \emph{vertical oscillation coefficients} $\osc_{\Omega}(B(p,r))$. Given a domain $\Omega \subset \He$ and a point $p \in \partial \Omega$, the number $\osc_{\Omega}(B(p,r))$ quantifies, in a scale and translation invariant way, how far $\Omega$ is (locally) from satisfying \eqref{alternative}. The definition of the coefficients $\osc_{\Omega}(B(p,r))$ was inspired by the notion of \emph{vertical perimeter} recently introduced by Lafforgue and Naor in \cite[Section 4]{MR3273443}, and further studied by Naor and Young in \cite{NY}; Remark \ref{vPerRemark} for the definition. We postpone further details on the vertical oscillation coefficients to Section \ref{vosc}.

Here is the main theorem of the paper:
\begin{thm}\label{mainIntro} Let $\Gamma \subset \He$ be an intrinsic Lipschitz graph, and let $\Omega$ be one of the components of $\He \setminus \Gamma$. Assume that there is a finite constant $C > 0$ such that
\begin{equation}\label{mainAssumption} \int_{0}^{\infty} \osc_{\Omega}(B(p,r)) \, \frac{dr}{r} \leq C, \qquad p \in \partial \Omega. \end{equation}
Then $R$ is bounded on $L^{2}(\calH^{3}|_{\Gamma})$.
\end{thm}
In general, we do not know how reasonable the assumption \eqref{mainAssumption} is. It follows easily from the Rademacher theorem for intrinsic Lipschitz functions (and Corollary \ref{oscillationAndBetas} below) that $\osc_{\Omega}(B(p,r)) \to 0$ for $\calH^{3}$ almost every $p \in \Gamma$ as $r \searrow 0$. But we have no quantitative estimates for $\osc_{\Omega}(B(p,r))$ if nothing better than intrinsic Lipschitz regularity is assumed of $\Gamma$; see Section \ref{problems} for a concrete question in this vein. However, we can complement Theorem \ref{mainIntro} with the following application:
\begin{thm}\label{main2Intro} Let $\phi \colon \W \to \R$ be an intrinsic Lipschitz function that satisfies the following H\"older regularity in the vertical direction:
\begin{equation}\label{holder} |\phi(y,t) - \phi(y,s)| \leq H|t - s|^{(1 + \tau)/2}, \qquad |s - t| \leq 1, \end{equation}
and
\begin{equation}\label{holder2} |\phi(y,t) - \phi(y,s)| \leq H|t - s|^{(1 - \tau)/2}, \qquad |s - t| > 1, \end{equation}
where $H \geq 1$ and $0 < \tau \leq 1$. Then $R$ is bounded on $L^{2}(\calH^{3}|_{\Gamma_{\phi}})$.
\end{thm}
It is well-known that intrinsic Lipschitz functions are always $1/2$-H\"older continuous in the vertical direction. So, Theorem \ref{main2Intro} states that an $\epsilon$ of additional regularity in this one direction yields the $L^{2}$-boundedness of $R$ on $\Gamma_{\phi}$.

\subsection{Vertical oscillation and $\beta$-numbers} A fundamental concept in the theory of quantitative rectifiability in $\R^{n}$ is the \emph{$\beta$-number}, first introduced by Jones in \cite{MR1069238}, then further developed by David and Semmes \cite{DS1}, and later applied by too many authors to begin acknowledging here. It is no surprise that suitable variants of the $\beta$-numbers (see Section \ref{betaNumberSection} for definitions) can also be used to study quantitative rectifiability questions in $\He$, as well as higher dimensional Heisenberg groups. A few papers already doing so are \cite{CFO,CFO2,CL,2018arXiv180304819F,MR2789375,MR3456155,MR3512421}. Since we here introduce new coefficients related to the theory of quantitative rectifiability in $\He$, it is natural to ask: is there a connection to $\beta$-numbers? We investigate this matter in Sections \ref{betaNumberSection} and \ref{perimeterAndBetas}.

We only mention the key results here briefly and informally. First, the vertical oscillation coefficients of $\Omega$ are bounded from above by the ($L^{1}$-based) $\beta$-numbers of $\partial \Omega$ -- at least if $\partial \Omega$ is $3$-regular. This is the content of Corollary \ref{oscillationAndBetas}. Second, if $\partial \Omega$ is $3$-regular, and if the $\beta$-numbers associated to $\partial \Omega$ satisfy an \emph{$L^{p}$-Carleson packing condition}, see \eqref{carlesonBeta}, then the $L^{p}$-variant of the vertical perimeter of $\Omega$ inside balls $B(q,r)$, $q \in \partial \Omega$, is bounded by the usual (horizontal) perimeter of $\Omega$ in $B(q,r)$. This is Corollary \ref{verticalPerimeterCorollary}.

This result should be contrasted with the work of Naor and Young in higher dimensional Heisenberg groups: in \cite[Proposition 41]{NY}, they prove that if $\Omega \subset \He^{n}$, $n \geq 2$, is an intrinsic Lipschitz domain, then the $L^{2}$-vertical perimeter of $\Omega$ in balls centred at $\partial \Omega$ is automatically bounded by the horizontal perimeter -- without any reference to $\beta$-numbers. Then, at the very end of \cite{NY}, see also \cite[Remark 4]{NY}, the authors mention showing in a forthcoming paper \cite{NY2} that a similar inequality fails for the $L^{2}$-vertical perimeter in $\He^1=\He$, but holds for the $L^{p}$-vertical perimeter for some $p > 2$ (specifically, the authors mention $p = 4$). If this is the case, then, according to Corollary \ref{verticalPerimeterCorollary}, one cannot expect the $\beta$-numbers of intrinsic Lipschitz graphs to satisfy an $L^{2}$-Carleson packing condition. This is in contrast to the situation in $\R^{n}$, where the $\beta$-numbers on Lipschitz graphs do satisfy an $L^{2}$-Carleson packing condition, see \cite[(C3)]{DS1}. 

\section{Preliminaries}

In this section, we collect essential notions related to the algebraic and the metric structure of the first Heisenberg group $\He$, and we recall the definition and basic properties of intrinsic Lipschitz graphs over vertical planes in $\He$. For a more thorough introduction to these subjects, we refer the reader to \cite{MR2312336,MR3587666} and the references therein.

\subsection{Right and left invariant vector fields}\label{s:vector_fields} Recall from the introduction that $X$ and $Y$ denote the standard left invariant vector fields on $\He$ defined in \eqref{eq:LeftInvVfd}. We will also work with their  right invariant counterparts
\begin{displaymath} \widetilde{X} = \partial_{x} + \tfrac{y}{2}\partial_{t} \quad \text{and} \quad \widetilde{Y} = \partial_{y} - \tfrac{x}{2}\partial_{t}. \end{displaymath}
We define the {left and right (horizontal) gradients} of $\phi \in \mathcal{C}^{1}(\R^{3})$ as the $2$-vectors
\begin{displaymath} \nabla_{\He}\phi = (X\phi,Y\phi) \quad \text{and} \quad \widetilde{\nabla}_{\He}\phi = (\widetilde{X}\phi,\widetilde{Y}\phi). \end{displaymath}
For $V = (V_{1},V_{2}) \in \mathcal{C}^{1}(\R^{3},\R^{2})$, we define the {left and right divergences} as the functions
\begin{displaymath} \Div V := XV_{1} + YV_{2} \in \mathcal{C}^{0}(\R^{3}) \quad \text{and} \quad \widetilde{\Div} V := \widetilde{X}V_{1} + \widetilde{Y}V_{2} \in \mathcal{C}^{0}(\R^{3}). \end{displaymath}
For $V,W \in \mathcal{C}^{1}(\R^{3},\R^{2})$, we define the "inner product"
\begin{displaymath} \langle V,W \rangle := V_{1}W_{1} + V_{2}W_{2} \in \mathcal{C}^{1}(\R^{3}). \end{displaymath}
Finally, we denote the left and right sub-Laplacians as
\begin{displaymath} \bigtriangleup_{\He} := XX + YY \quad \text{and} \quad \widetilde{\bigtriangleup_{\He}} := \widetilde{X}\widetilde{X} + \widetilde{Y}\widetilde{Y}. \end{displaymath}

\subsection{Metric structure}
Various left invariant distance functions on $\He$ are commonly used in the literature, for instance the standard {sub-Riemannian distance} or the {Kor\'{a}nyi metric} given in \eqref{eq:Kor}. The choice of metric that we are going to use in the following is motivated by the divergence theorem (Theorem \ref{divergenceTheorem}), which holds for the spherical Hausdorff measure $\mathcal{S}^3$ with respect to the metric
\begin{equation}\label{eq:metric}
d:\He \times \He \to [0,+\infty),\quad d(p,q):= \|q^{-1}\cdot p\|,
\end{equation}
where
\begin{displaymath}
\|(x,y,t)\|:= \max\{|(x,y)|, 2 \sqrt{|t|}\}.
\end{displaymath}
However, every left invariant metric on $\He$ that is continuous with respect the Euclidean topology on $\mathbb{R}^3$ and  homogeneous with respect to the one-parameter family of \emph{Heisenberg dilations} $(\delta_{\lambda})_{\lambda>0}$
\begin{displaymath}
\delta_{\lambda}:\He \to \He,\quad \delta_{\lambda}(x,y,t):=(\lambda x, \lambda y,\lambda^2 t)
\end{displaymath}
is bi-Lipschitz equivalent to the metric $d$; this applies in particular to the Kor\'{a}nyi distance $d_{Kor}$. Unless otherwise stated, all metric concepts such as balls $B(p,r)$, diameters, and Hausdorff measures will be defined using the metric $d$.


\subsection{Intrinsic Lipschitz graphs}\label{ss:intrLip}
Let $\mathbb{W}$ be a vertical subgroup with complementary horizontal subgroup $\mathbb{V}$. Any point $p\in \He$ can be written as $p= w \cdot v$ for uniquely given $w\in \W$ and $v\in \V$. We write $w=:\pi_{\W}(p)$ and call it the \emph{vertical projection of $p$ to $\W$}; similarly, we denote the \emph{horizontal projection} by $v=\pi_{\V}(p)$.
These projections have been studied in connection with uniform rectifiability problems in the Heisenberg group, see for example \cite{CFO,2018arXiv180304819F}.

\begin{definition}
A function $\phi: \W \to \V$ is \emph{intrinsic $L$-Lipschitz} if
\begin{equation}\label{eq:intrLip}
\|\pi_{\V}\left(\Phi(w')^{-1}\Phi(w)\right)\|\leq L \left\|\pi_{\mathbb{W}}\left(\Phi(w')^{-1}\Phi(w)\right) \right\|,\quad\text{for all }w,w'\in\mathbb{W},
\end{equation}
where $\Phi:\mathbb{W}\to \He$ denotes the \emph{graph map} $\Phi(w)=w\cdot \phi(w)$. The \emph{intrinsic graph} of $\phi$ is
\begin{displaymath}
\Gamma_{\phi}:= \{w\cdot \phi(w):\; w\in \W\}=\Phi(\W).
\end{displaymath}
\end{definition}
The term "intrinsic" refers to the fact that if $\phi$ is an intrinsic $L$-Lipschitz function, then, for all $p\in \He$ and $r>0$, also $\tau_p(\delta_r(\Gamma_{\phi}))$ is an intrinsic graph of an intrinsic $L$-Lipschitz function. According to \cite[Remark 2.6]{CFO}, an intrinsic $L$-Lipschitz graph over an arbitrary vertical plane can be mapped to an intrinsic $L$-Lipschitz graph over the $(y,t)$-plane by an isometry of the form
\begin{displaymath}
R_{\theta}:\He \to \He,\quad R_{\theta}(x,y,t) := (x\cos \theta  + y\sin \theta , -x\sin \theta  +y\cos \theta , t).
\end{displaymath}
Since moreover the (complexified) kernel of the Heisenberg Riesz transform satisfies
\begin{displaymath}
(XG- i YG)\circ R_{\theta} = e^{i\theta} (XG - i YG),
\end{displaymath}
we may without loss of generality assume in the following that $\W$ is the $(y,t)$-plane and $\V$ is the $x$-axis. For this choice, we have
\begin{displaymath}
\pi_{\V}(x,y,t)=(x,0,0)\quad \text{and} \quad \pi_{\mathbb{W}}(x,y,t) = \left(0,y,t+\tfrac{1}{2}xy\right),\quad\text{for all }(x,y,t)\in \He.
\end{displaymath}
Moreover, the map $(x,0,0) \mapsto x$, provides an isometric isomorphism between $(\V,\cdot,d)$ and $(\R,+,|\cdot|)$, and under this identification of $\V$ with $\R$, the intrinsic Lipschitz condition \eqref{eq:intrLip} is equivalent to
\begin{displaymath}
|\phi(0,y,t)-\phi(0,y',t')| \leq L \left\|\pi_{\mathbb{W}}\left(\Phi(0,y',t')^{-1}\Phi(0,y,t)\right) \right\|,\quad\text{for all }(y,t),(y',t')\in \R^2.
\end{displaymath}
The subgroup $(\W,\cdot)$ is isomorphic to $(\R^2,+)$, and the map $(0,y,t)\mapsto (y,t)$ pushes the measure $\mathcal{H}^3|_{\W}$ forward to $c\mathcal{L}^2$ on $\R^2$, for a constant $0<c<\infty$. As mentioned in the introduction, an intrinsic Lipschitz function $\phi: \W \to \V$ possesses an \emph{intrinsic gradient} $\nabla^{\phi}\phi$ at $\mathcal{H}^3$ almost every point of $\W$.
In analogy with the behavior of Euclidean Lipschitz functions, if $\phi: \W \to \V$ is intrinsic Lipschitz, then
\begin{displaymath} \|\nabla^{\phi}\phi\|_{L^{\infty}(\calH^3|_{\W})} < \infty, \end{displaymath}
by \cite[Proposition 4.4]{MR3168633}. More information about intrinsic gradients is collected for instance in \cite{MR3587666} and in \cite[Section 4.2]{CFO}.

\section{Vertical oscillation coefficients}\label{vosc}

In this section, we define and study the main new concept of the paper, the \emph{vertical oscillation coefficients}. These coefficients are derived from the recent notion of \emph{vertical perimeter}, due to Lafforgue and Naor \cite[Definition 4.2]{MR3273443} (see also \cite[(28)]{NY}):
\begin{definition}[Vertical perimeter] Let $\Omega,U \subset \He$ be Lebesgue measurable sets, and let $s > 0$ be a scale. The \emph{vertical perimeter of $\Omega$ relative to $U$ at scale $s$} is the quantity
\begin{displaymath} \w_{\Omega}(U)(s) := \int_{U} |\chi_{\Omega}(p) - \chi_{\Omega}(p \cdot (0,0,s^{2}))| \, dp. \end{displaymath}
\end{definition}
Here and in the following, $dp$ refers to integration with respect to  Lebesgue measure $\mathcal{L}^3$ on $\mathbb{R}^3$, which agrees up to a multiplicative constant with $\mathcal{H}^4$.
\begin{remark}\label{vPerRemark} Having first defined the vertical perimeter $\w_{\Omega}(U)(s)$ at a fixed scale $s > 0$, Lafforgue and Naor \cite[(70)]{MR3273443} and Naor and Young \cite[Section 2.2]{NY} proceed to define the \emph{$L^{2}$-vertical perimeter of $\Omega$} as the $L^{2}(ds/s)$-norm of the function $s \mapsto \w_{\Omega}(\He)/s$. More generally, for $p \geq 1$ and an open set $U \subset \He$, one can consider (as in \cite[(68)]{NY} for example)  the \emph{$L^{p}$-vertical perimeter of $\Omega$ in $U$}:
\begin{displaymath} \wp_{\Omega,p}(U) := \left\|s \mapsto \frac{\w_{\Omega}(U)(s)}{s}\right\|_{L^{p}(ds/s)} = \left(\int_{0}^{\infty} \left(\frac{\w_{\Omega}(U)(s)}{s} \right)^{p} \, \frac{ds}{s} \right)^{1/p}. \end{displaymath}
It would be interesting to know if the $L^{p}$-vertical perimeter of $\Omega$ -- for some $p \geq 1$, and for an intrinsic Lipschitz domain $\Omega$, say -- can be related to the boundedness of the Heisenberg Riesz transform on $L^{2}(\calH^{3}|_{\partial \Omega})$.
\end{remark}

We now define the vertical oscillation coefficients:
\begin{definition}[Vertical oscillation coefficients]\label{oscillationCoeff} Let $\Omega \subset \He$ be a Lebesgue measurable (typically open) set, and let $B(p,r) \subset \He$ be a ball. We define
\begin{displaymath} \osc_{\Omega}(B(p,r)) := \fint_{0}^{r} \frac{\w_{\Omega}(B(p,r))(s)}{r^{4}} \, ds.  \end{displaymath}
\end{definition}

We examine the basic properties of the oscillation coefficients in the next lemma:
\begin{lemma}\label{basicProp} There is an absolute constant $C \geq 1$ such that $\osc_{\Omega}(B(p,r)) \leq C$ for all Lebesgue measurable sets $\Omega \subset \He$, and all balls $B(p,r) \subset \He$. The vertical oscillation coefficients are approximately monotone in the following sense: if $B(p_{1},r_{1}) \subset B(p_{2},r_{2}) \subset \He$ are two balls with $r_{2} \leq C_{1}r_{1}$, then
\begin{equation}\label{approxMonotone} \osc_{\Omega}(B(p_{1},r_{1})) \lesssim_{C_{1}} \osc_{\Omega}(B(p_{2},r_{2})). \end{equation}
Finally, the vertical oscillation coefficients are invariant with respect to dilations and left translations in the following sense:
\begin{equation}\label{form1} \osc_{\delta_{t}(q \cdot \Omega)}(B(\delta_{t}(q \cdot p),tr)) = \osc_{\Omega}(B(p,r)), \qquad t > 0, \; q \in \He. \end{equation}
\end{lemma}

\begin{proof} To prove the first claim, observe that $\w_{\Omega}(B(p,r))(s) \leq 2\calH^{4}(B(p,r)) \sim r^{4}$ for all $0 \leq s \leq r$, so
\begin{displaymath} \osc_{\Omega}(B(p,r)) \lesssim \fint_{0}^{r} \frac{r^{4}}{r^{4}} \, ds = 1. \end{displaymath}
The approximate monotonicity property \eqref{approxMonotone} follows immediately from the inequality $\w_{\Omega}(B(p_{1},r_{1}))(s) \leq \w_{\Omega}(B(p_{2},r_{2}))(s)$, valid for all $s > 0$.

The left-invariance $\osc_{q \cdot \Omega}(B(q \cdot p,r)) = \osc_{\Omega}(B(p,r))$ of the vertical oscillation coefficients follows from the evident left-invariance of the vertical perimeter, so we assume that $p = q = 0$ and prove that
\begin{displaymath} \osc_{\delta_{t}(\Omega)}(B(0,tr)) = \osc_{\Omega}(B(0,r)), \qquad t > 0. \end{displaymath}
To see this, we start by expanding
\begin{align*} \osc_{\delta_{t}(\Omega)}(B(0,tr)) & = \frac{1}{(tr)^{5}} \int_{0}^{tr} \w_{\delta_{t}(\Omega)}(B(0,tr))(s) \, ds\\
& = \frac{1}{(tr)^{5}} \int_{0}^{tr} \int_{B(0,tr)} |\chi_{\delta_{t}(\Omega)}(p) - \chi_{\delta_{t}(\Omega)}(p \cdot (0,0,s^{2}))| \, dp \, ds \end{align*}
Then, we make the change of variables $p \mapsto \delta_{t}(q)$, and finally $s \mapsto ut$:
\begin{displaymath} \osc_{\delta_{t}(\Omega)}(B(0,tr)) = \frac{1}{r^{5}} \int_{0}^{r} \int_{B(0,r)} |\chi_{\Omega}(q) - \chi_{\Omega}(q \cdot (0,0,u^{2}))| \, dq \, du = \osc_{\Omega}(B(0,r)). \end{displaymath}
This completes the proof. \end{proof}

\begin{remark} The previous lemma says that $\osc_{\Omega}(B(p,r)) \lesssim 1$ no matter what $\Omega$ looks like. If $\Omega$ is the sub- or super-graph of an intrinsic Lipschitz function satisfying better than $\tfrac{1}{2}$-H\"older regularity in the vertical direction, then the oscillation coefficients of $\Omega$ have geometric decay. A more precise statement can be found in Lemma \ref{holderOscillation}.  \end{remark}

In connection with singular integrals, the vertical oscillation coefficients will enter through the next lemma:
\begin{lemma}\label{mainLemma} Let $\Omega \subset \He$ be a Lebesgue measurable set. Let $p \in \He$, $r > 0$, and let $\psi \in \mathcal{C}^{1}(\R^{3})$ with $\spt \psi \subset B(p,r)$. Then,
\begin{equation}\label{form2} \left| \frac{1}{r^{4}} \int_{\Omega}  \partial_{t} \psi(p) \, dp \right| \lesssim \|\partial_{t}\psi\|_{\infty}\osc_{\Omega}(B(p,10r)),  \end{equation}
where $\partial_{t}\psi$ is the derivative of $\psi$ with respect to the third variable.
\end{lemma}

\begin{proof} We start by reducing to the case $B(p,r) = B(0,1)$. So, assume that \eqref{form2} holds for every Lebesgue measurable set $\Omega$ and all $\psi \in \mathcal{C}^{1}(\R^{3})$ with $\spt \psi \subset B(0,1)$ and with $\osc_{\Omega}(B(0,10))$ on the right hand side. Then, if $\psi \in \mathcal{C}^{1}(\R^{3})$ with $\spt \psi \subset B(p,r)$, we consider the function $\psi_{p,r} = \psi \circ \tau_{p} \circ \delta_{r} \in \mathcal{C}^{1}(\R^{3})$ with $\spt \psi_{p,r} \subset B(0,1)$. It follows that
\begin{align*} \left|\frac{1}{r^{4}} \int_{\Omega} \partial_{t} \psi(q) \, dq \right| & = \left| \int_{\delta_{1/r}(p^{-1} \cdot \Omega)} \partial_{t} \psi (\delta_{r}(p \cdot q)) \, dq \right| = \left| \int_{\delta_{1/r}(p^{-1} \cdot \Omega)} r^{-2} \partial_{t}\psi_{p,r}(q) \, dq \right|\\
& \lesssim \frac{\|\partial_{t} \psi_{p,r} \|_{\infty}}{r^{2}} \osc_{\delta_{1/r}(p^{-1} \cdot \Omega)}(B(0,10))\\
& = \|\partial_{t} \psi\|_{\infty} \osc_{\Omega}(B(p,10r)), \end{align*}
using Lemma \ref{basicProp} in the last equation.

It remains to prove the case $B(p,r) = B(0,1)$, so fix $\psi \in \mathcal{C}^{1}(\R^{3})$ with $\spt \psi \subset B(0,1)$. By Fubini's theorem, we can write
\begin{equation}\label{form4} \int_{\Omega} \partial_{t} \psi(q) \, dq = \int_{\mathcal{L}} \int_{\ell} \partial_{t} \psi(q)\chi_{\Omega}(q) \, d\calH^{1}_{E}(q) \, d\eta(\ell), \end{equation}
where $\mathcal{L}$ stands for the collection of vertical lines, $\eta$ is two-dimensional Lebesgue measure on $\R^{2}$ (which is used to parametrise $\mathcal{L}$), and $\calH^{1}_{E}$ denotes the $1$-dimensional Hausdorff measure with respect to the Euclidean distance. Next, we note that if $\ell \in \mathcal{L}$ is a fixed line, then
\begin{equation}\label{form3} \int_{\ell} \partial_{t} \psi(q) \, d\calH^{1}_{E}(q) = 0. \end{equation}
Now, let $Q := [-5,5]^{2} \times [-2,-1] \subset B(0,10)$. Note that whenever $\ell \in \mathcal{L}$ is a line with non-zero contribution in \eqref{form4}, then $\ell \cap B(0,1) \neq \emptyset$, and in particular
\begin{displaymath} \mathcal{H}^{1}_{E}(\ell \cap Q) = 1. \end{displaymath}
Then, use \eqref{form4}-\eqref{form3} to write
\begin{align*} \left|\int_{\Omega} \partial_{t} \psi(q) \, dq \right| & = \left| \int_{\mathcal{L}} \int_{\ell \cap Q} \int_{\ell} \partial_{t}\psi(q)[\chi_{\Omega}(q) - \chi_{\Omega}(p)] \, d\calH^{1}_{E}(q) \, d\calH^{1}_{E}(p) \, d\eta(\ell) \right|\\
& \leq \|\partial_{t} \psi\|_{\infty} \int_{\mathcal{L}} \int_{\ell \cap Q} \int_{\ell \cap B(0,1)} |\chi_{\Omega}(q) - \chi_{\Omega}(p)| \, d\calH^{1}_{E}(q) \, d\calH^{1}_{E}(p) \, d\eta(\ell). \end{align*}
Next, for $\ell \in \mathcal{L}$ and $p \in \ell \cap Q$ fixed, we make the change of variable $q \mapsto p \cdot (0,0,s)$ in the innermost integral: since $q \in \ell \cap B(0,1)$ and $p \in \ell \cap Q$, we note that $s \in [0,3]$. This leads to
\begin{align*} \left|\int_{\Omega} \partial_{t} \psi(q) \, dq \right| & \leq \|\partial_{t} \psi\|_{\infty} \int_{\mathcal{L}} \int_{\ell \cap Q} \int_{0}^{3} |\chi_{\Omega}(p \cdot (0,0,s)) - \chi_{\Omega}(p)| \, ds \, d\calH^{1}_{E}(p) \, d\eta(\ell)\\
& \leq \|\partial_{t} \psi\|_{\infty} \int_{0}^{3} \int_{\mathcal{L}} \int_{\ell \cap B(0,10)} |\chi_{\Omega}(p \cdot (0,0,s)) - \chi_{\Omega}(p)| \, d\calH^{1}_{E}(p) \, d\eta(\ell) \, ds\\
& \lesssim \|\partial_{t}\psi\|_{\infty} \int_{0}^{\sqrt{3}} \w_{\Omega}(B(0,10))(s) \, ds \lesssim \|\partial_{t} \psi\|_{\infty} \osc_{\Omega}(B(0,10)).  \end{align*}
This completes the proof. \end{proof}

\subsection{Vertical oscillation vs. vertical $\beta$-numbers}\label{betaNumberSection} Given a set $E \subset \He$ and a ball $B(q,r) \subset \He$, we recall from \cite[Definition 3.3]{CFO} the following \emph{vertical $\beta$-number of $E$ in $B(q,r)$, $q\in E$}:
\begin{displaymath} \beta_{E,\infty}(B(q,r)) := \inf_{\W,z} \sup_{x \in B(q,r) \cap E} \frac{\dist(x,z \cdot \W)}{r}, \end{displaymath}
where the $\inf$ runs over all vertical subgroups $\W \subset \He$, and all points $z \in \He$. More generally, one can consider the following $L^{p}$-variants:
\begin{displaymath} \beta_{E,p}(B(q,r)) := \inf_{\W,z} \left( \frac{1}{r^{3}} \int_{B(q,r)\cap E} \left( \frac{\dist(x,z \cdot \W)}{r} \right)^{p} \, d\calH^{3}(x) \right)^{1/p}, \qquad 1 \leq p < \infty, \end{displaymath}
assuming that $E$ has locally finite $3$-dimensional measure. If $E$ happens to be $3$-regular, then the $\beta_{E,p}$-numbers are essentially monotone in $p$:
\begin{displaymath} \beta_{E,p_{1}}(B(q,r)) \lesssim \beta_{E,p_{2}}(B(q,r)), \qquad q \in E, \: 1 \leq p_{1} \leq p_{2} \leq \infty. \end{displaymath}
The next theorem shows that the vertical oscillation coefficients of $\Omega$ are always bounded by the $\beta_{E,\infty}$-numbers of $\partial \Omega$, and also "almost" bounded from above by the $\beta_{E,1}$-numbers of $\partial \Omega$. After this statement concerning general domains $\Omega$, we will give a corollary to domains with $3$-regular boundaries: in this case the word "almost" above can be omitted.

\begin{thm}\label{oscillationBetaComparison} Let $\Omega \subset \He$ be an open set such that $\partial \Omega$ has locally finite $3$-dimensional measure, and let $p \in \partial \Omega$ and $r > 0$. Then, for any
$p \in \partial \Omega$, and $0 < s \leq r$,
\begin{equation}\label{form17} \frac{\w_{\Omega}(B(p,r))(s)}{r^{4}} \lesssim_{\epsilon} \inf_{\W,z} \left[ \frac{1}{r^{3}} \int_{B(p,12r) \cap \partial \Omega} \frac{d(q,z \cdot \W)}{12r} \, d\calH^{3}(q) + \epsilon\left(\sup_{q \in B(p,12r) \cap \partial \Omega} \frac{d(q,z \cdot \W)}{12r} \right) \right] \end{equation}
for any non-decreasing function $\epsilon \colon \R_{+} \to \R_{+}$ such that $\epsilon(\delta) \to 0$ as $\delta \to 0$.
\end{thm}

The same estimate for the vertical oscillation coefficient $\osc_{\Omega}(B(p,r))$ follows immediately by taking the average over $s \in (0,r]$ on the left hand side; we will however need the sharper result later, in Section \ref{perimeterAndBetas}. Note also that the quantity on the right hand side of \eqref{form17} looks like
\begin{displaymath} \beta_{\partial \Omega,1}(B(p,12r)) + \epsilon[\beta_{\partial \Omega,\infty}(B(p,12r))], \end{displaymath}
but can be sometimes larger, as only one choice of $z,\W$ is made on the right hand side of \eqref{form17}. The quantities on both sides of the inequality \eqref{form17} are invariant under scaling and translation, so we may assume that $p = 0$ and $r = 1$. We start the proof with the following simple lemma:
\begin{lemma}\label{lemma1} Let $\Omega \subset \He$ be an open set. Let $H \subset \He$ be a vertical half-space, that is, a half-space bounded by the translate of some vertical subgroup. Then,
\begin{displaymath} \w_{\Omega}(B(0,1))(s) \leq 2\calH^{4}([\Omega \bigtriangleup H] \cap B(0,3)), \qquad 0 < s \leq 1. \end{displaymath}
\end{lemma}

\begin{proof} Let $0 \leq s \leq 1$. Note that $\chi_{H}(q) = \chi_{H}(q \cdot (0,0,s^{2}))$ for all $q \in \He$. Hence,
\begin{align*} \w_{\Omega}(B(0,1))(s) & \leq \int_{B(0,1)} |\chi_{\Omega}(q) - \chi_{H}(q) + \chi_{H}(q \cdot (0,0,s^{2})) - \chi_{\Omega}(q \cdot (0,0,s^{2}))| \, dq\\
& \leq 2\int_{B(0,3)} |\chi_{\Omega}(q) - \chi_{H}(q)| \, dq  = 2\calH^{4}([\Omega \bigtriangleup H] \cap B(0,3)). \end{align*}
This is the desired estimate. \end{proof}
Now, to conclude the proof of Theorem \ref{oscillationBetaComparison}, it suffices to show (after scaling $\Omega$ by $1/3$) that there exists a vertical half-space $H \subset \He$ such that
\begin{equation}\label{form18} \calH^{4}([\Omega \bigtriangleup H] \cap B(0,1)) \lesssim_{\epsilon}  \inf_{\W,z} \left[ \int_{B(0,4) \cap \partial \Omega} d(q,z \cdot \W) \, d\calH^{3}(q) + \epsilon\left(\sup_{q \in B(0,4) \cap \partial \Omega} d(q,z \cdot \W) \right) \right]. \end{equation}
Further, to prove \eqref{form18}, we may assume that if $P := z \cdot \W$ is a vertical plane minimising the right hand side in \eqref{form18}, then
\begin{equation}\label{form19} \sup_{q \in B(0,4) \cap \partial \Omega} d(q,P) \leq \delta := 10^{-10}. \end{equation}
Indeed, \eqref{form18} is clear if the converse of \eqref{form19} holds. In particular, since $0 = p \in \partial \Omega$, we see that $P$ is at distance at most $\delta$ to the $yt$-plane. For slight notational convenience, we will in fact assume that $P = \{(0,y,t) : y,t \in \R\}$. Now, under the assumption \eqref{form19}, we will actually show that there exists a vertical half-space $H \subset \He$ (not necessarily bounded by $P$) such that
\begin{equation}\label{form20} \calH^{4}([\Omega \bigtriangleup H] \cap B(0,1)) \lesssim \int_{B(0,4) \cap \partial \Omega} d(q,P) \, d\calH^{3}(q). \end{equation}
So, the $L^{1}$-based $\beta$-number of $\partial \Omega$ dominates the vertical oscillation of $\Omega$ under the \emph{a priori} assumption that the $L^{\infty}$-based $\beta$-number is sufficiently small. We now choose $H$. We denote the (closed) half-spaces bounded by $P$ by
\begin{displaymath} \He_{+} := \{(x,y,t) : x \geq 0\} \quad \text{and} \quad \He_{-} := \{(x,y,t) : x \leq 0\}. \end{displaymath}
Write $U_{+},U_{-}$ for the connected components of $B(0,4) \setminus P(\delta)$, where $P(\delta)$ is the closed $\delta$-neighbourhood of $P$, with
\begin{displaymath} U_{+} \subset \He_{+} \quad \text{and} \quad U_{-} \subset \He_{-}. \end{displaymath}
By \eqref{form19}, we may infer that either $U_{+} \subset \Omega$ or $U_{-} \cap \Omega = \emptyset$, and similarly either $U_{-} \subset \Omega$ or $U_{-} \cap \Omega = \emptyset$. The definition of $H$ depends on which of these cases occur:
\begin{itemize}
\item[(a)] If $U_{-} \subset \Omega$ and $U_{+} \cap \Omega = \emptyset$, let $H := \He_{-}$.
\item[(b)] If $U_{-} \cap \Omega = \emptyset$ and $U_{+} \subset \Omega$, let $H := \He_{+}$.
\item[(c)] If $U_{+},U_{-} \subset \Omega$, let $H$ be any vertical half-space containing $B(0,4)$.
\item[(d)] If $U_{+} \cap \Omega = \emptyset = U_{-} \cap \Omega$, let $H$ be any vertical half-space with $H \cap B(0,4) = \emptyset$.
\end{itemize}
The point of these choices is that always
\begin{equation}\label{form21} [\Omega \bigtriangleup H] \cap B(0,4) \subset P(\delta), \end{equation}
as one may easily verify.

\begin{figure}[h!]
\begin{center}
\includegraphics[scale = 0.3]{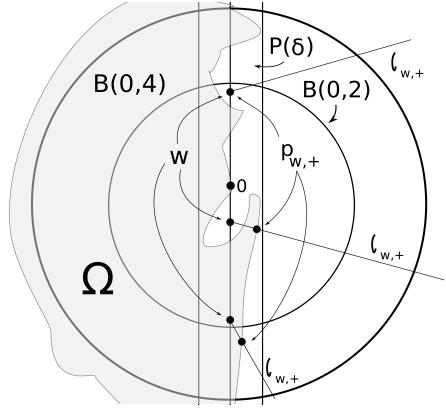}
\caption{Various concepts in the proof of Theorem \ref{oscillationBetaComparison}. Scenario (a) is depicted.}\label{fig1}
\end{center}
\end{figure}

We claim that \eqref{form20} holds for the choice of $H$ above. To see this, we need additional notation. For $w \in P$, let
\begin{displaymath} \ell_{w} := \{w \cdot (x,0,0) : x \in \R\} \end{displaymath}
be the left translate of the $x$-axis passing through $w$. We also define the half-lines
\begin{displaymath} \ell_{w,+} := \ell_{w} \cap \He_{+} \quad \text{and} \quad \ell_{w,-} := \ell_{w} \cap \He_{-}, \end{displaymath}
see Figure \ref{fig1}. To prove \eqref{form20}, we study separately the parts of $[\Omega \bigtriangleup H] \cap B(0,1)$ inside $\He_{-}$ and $\He_{+}$. These investigations are symmetrical, so we restrict attention to $\He_{+}$. For notational convenience, we write $B(0,s) \cap \He_{+} := B_{+}(0,s)$ in the sequel.
We will apply the general integration estimate
\begin{equation}\label{eq:Fubini}
\calH^4(A) \sim \int_P \calH^1(A\cap \ell_w) \, dw,\quad A\subset \He \text{ Borel}.
\end{equation}
Here "$dw$" refers to the $3$-dimensional Hausdorff measure on $P$, which coincides (up to a constant) with Lebesgue measure on $P$. To establish formula \eqref{eq:Fubini}, recall that $\calH^4$ agrees up to a multiplicative constant with the $3$-dimensional Lebesgue measure and the transformation $\Phi: \R^2 \times \R \to \He$, $\Phi((w_1,w_2),s)= (0,w_1,w_2) \cdot (s,0,0)$ has Jacobian determinant equal to $1$. Hence,
\begin{equation}\label{eq:Fub1}
\calH^4(A) \sim \int_{\R^2} \int_{-\infty}^{\infty} \chi_A(\Phi(w,s)) \, ds \, dw.
\end{equation}
Next, for every $w\in P$,  the map $s \mapsto \Phi(w,s) = w \cdot (s,0,0)$ is an isometry between $(\R,|\cdot|)$ and $(\ell_w,d)$, and thus we find that
\begin{equation}\label{eq:Fub2}
\int_{-\infty}^{\infty}  \chi_A(\Phi(w,s)) \, ds = \int_{\ell_{w}} \chi_A(q) \;d\mathcal{H}^1(q) = \calH^{1}(A \cap \ell_{w}).
\end{equation}
These facts together prove \eqref{eq:Fubini}. Applied to the set $A= [\Omega \bigtriangleup H] \cap B_{+}(0,1)$, this formula then yields
\begin{equation}\label{form26} \calH^{4}([\Omega \bigtriangleup H] \cap B_{+}(0,1)) \lesssim \int_{P\cap B(0,2)} \calH^{1}([\Omega \bigtriangleup H] \cap \ell_{w,+} \cap B(0,2)) \, dw.  \end{equation}
Here, the integration is restricted to $P\cap B(0,2)$ as $\Phi(w,s)$, $w\in P$, can lie in $B(0,1)$ only if $|s|\leq 1$, and in that case $d(\Phi(w,s),0)\geq d(w,0)- d(0,(s,0,0))>1$ if $w\in P\setminus B(0,2)$; in other words, the lines $\ell_{w}$ with $w \in P \setminus B(0,2)$ avoid $B(0,1)$.
Now, we fix $w \in P\cap B(0,2)$, and we will establish a suitable pointwise bound for the integrand in \eqref{form26}. To this end,
\begin{itemize}
\item if $\ell_{w,+} \cap \partial [\Omega \bigtriangleup H] \cap B(0,4) = \emptyset$, set $p_{w,+} := w$.
\item if $\ell_{w,+} \cap \partial [\Omega \bigtriangleup H] \cap B(0,4) \neq \emptyset$, let
\begin{displaymath} p_{w,+} := \max \left[\ell_{w,+} \cap \partial [\Omega \bigtriangleup H] \cap B(0,4)\right], \end{displaymath}
where the $\max$ refers to the only natural ordering on $\ell_{w,+}$
\end{itemize}
Then, by \eqref{form21}, we have in both cases
\begin{equation}\label{form34} p_{w,+} \in \ell_{w,+} \cap P(\delta) \subset P(\delta) \cap B(0,3), \qquad w \in P \cap B(0,2). \end{equation}
(If $w$ is sufficiently close to $\partial B(0,2)$, then it may happen that $\ell_{w,+} \cap P(\delta) \not\subset B(0,2)$, see Figure \ref{fig1}. However, $\delta > 0$ has been chosen so small that the second inclusion in \eqref{form34} holds.) Next, we define
\begin{displaymath} h_+(w) := \dist(p_{w,+},P), \qquad w \in P\cap B(0,2). \end{displaymath}
The "suitable pointwise bound" for the integrand in \eqref{form26} is the following:
\begin{equation}\label{form22} \calH^{1}([\Omega \bigtriangleup H] \cap \ell_{w,+} \cap B(0,2)) \leq h_+(w), \qquad w \in P \cap B(0,2). \end{equation}
In proving \eqref{form22}, we may evidently assume that
\begin{equation}\label{form24} [\Omega \bigtriangleup H] \cap \ell_{w,+} \cap B(0,2) \neq \emptyset. \end{equation}
Now, to prove \eqref{form22}, we will first argue that also
\begin{equation}\label{form27} [\Omega \bigtriangleup H]^{c} \cap \ell_{w,+} \cap B(0,4) \neq \emptyset.  \end{equation}
This will follow immediately once we manage to argue that
\begin{equation}\label{form25} U_{+} \subset [\Omega \bigtriangleup H]^{c}, \end{equation}
since evidently $\ell_{w,+} \cap U_{+} \neq \emptyset$.
The proof of \eqref{form25} depends on the scenario (a)-(d):
\begin{itemize}
\item[(a)] Here $U_{+} \cap \Omega = \emptyset$ and $H = \He_{-}$, so $U_{+} \subset \Omega^{c} \cap H^{c} \subset [\Omega \bigtriangleup H]^{c}$.
\item[(b)] Here $U_{+} \subset \Omega$ and $H = \He_{+}$, so $U_{+} \subset \Omega \cap H \subset [\Omega \bigtriangleup H]^{c}$.
\item[(c)] Here $U_{+} \subset \Omega$ and $B(0,4) \subset H$, so $U_{+} \subset \Omega \cap H \subset [\Omega \bigtriangleup H]^{c}$.
\item[(d)] Here $U_{+} \cap \Omega = \emptyset$ and $H \cap B(0,4) = \emptyset$, so $U_{+} \subset \Omega^{c} \cap H^{c} \subset [\Omega \bigtriangleup H]^{c}$.
\end{itemize}
We have now established \eqref{form25}, and hence \eqref{form27}. Combining \eqref{form24}-\eqref{form27}, we see that
\begin{displaymath} p_{w,+} = \max\left[ \ell_{w,+} \cap \partial [\Omega \bigtriangleup H] \cap B(0,4)\right] \end{displaymath}
is well-defined, and moreover
\begin{equation}\label{form31} [\Omega \bigtriangleup H] \cap \ell_{w,+} \cap B(0,2) \subset [w,p_{w,+}], \end{equation}
where $[w,p_{w,+}]$ stands for the (horizontal) line segment connecting $w$ and $p_{w,+}$. The point $p_{w,+}$ can be uniquely expressed as $p_{w,+} = w \cdot v_+$, where $v_{+}=(x_+,0,0)$ for some $x_+\geq 0$. Thus, we find by the definition of the metric $d$ that
\begin{displaymath}
x_{+} \leq \| \bar w^{-1} w v_{+}\| = d(wv_{+},\bar w) = d(p_{w,+},\bar w),\quad \text{for all }\bar w\in P.
\end{displaymath}
On the other hand, it holds that $d(p_{w,+},w)=x_{+}$. Hence
\begin{equation}\label{eq:dist_obs}
h_+(w)= \mathrm{dist}(p_{w,+},P)= d(p_{w,+},w)=  \calH^{1}([w,p_{w,+}]),
\end{equation}
where the last identity follows from the fact that $x \mapsto w \cdot (x,0,0)$ is an isometry from $(\R,|\cdot|)$ to $(\ell_w, d)$. We can now infer \eqref{form22} from \eqref{form31} and \eqref{eq:dist_obs}.

\medskip

Before proceeding further, we record that the function $h_+: P\cap B(0,2) \to \R$ is Borel, in fact even upper semicontinuous.
To see this, note that
$p_{w,+}$ is always contained in the compact set
\begin{displaymath} K := (P \cup \partial [\Omega \bigtriangleup H]) \cap \overline{B(0,3)} \end{displaymath}
for $w\in P \cap B(0,2)$ and consequently, also $h_+(P \cap B(0,2))$ is contained in the compact set $K':=\{\mathrm{dist}(p,P):\; p\in K\} \subset \R$.
If $h_+$ was not upper semicontinuous, there would exist $w\in P\cap B(0,2)$, $\varepsilon>0$, and a sequence $(w_n)_n\subseteq P \cap B(0,2)$ with
\begin{displaymath} \lim_{n\to \infty} w_n = w \quad \text{and} \quad \lim_{n \to \infty} h_+(w_n) > h_+(w). \end{displaymath}
We may assume that the limit on the right exists by the compactness of $K'$. Reducing to a further subsequence if necessary, we may assume that the sequence of points $p_{w_{n},{+}} = w_{n}\cdot (h_{+}(w_{n}),0,0)$ converges to a point $p = w \cdot v \in K$. Moreover,
\begin{equation}\label{form32}
h_+(w)< \lim_{k\to \infty} h_{{+}}(w_{n}) = \lim_{k\to \infty} \mathrm{dist}(p_{w_{n},+},P) = \mathrm{dist}(p, P).
\end{equation}
Since $p\in \ell_{w,+} \cap \partial[\Omega \bigtriangleup H]\cap B(0,4)$ (note that $p \notin P$ by \eqref{form32}),  this contradicts the maximality in the definition of $p_{w,+}$, and the proof of the upper semicontinuity of $h_{+}$ is complete.

\medskip

We now resume the proof of our goal \eqref{form20}.
Combining \eqref{form21} and \eqref{form22}, we have now established that
\begin{equation}\label{form23} \calH^{4}([\Omega \bigtriangleup H] \cap B_{+}(0,1)) \lesssim \int_{P\cap B(0,2)} h_+(w) \, dw = \int_{P\cap B(0,2)} \dist(p_{w,+},P) \, dw. \end{equation}
Noting that $p_{w,+} \in \partial \Omega \cap B(0,4)$ if $\dist(p_{w,+},P)\neq 0$, this conclusion is not too far from \eqref{form20} anymore. To arrive at \eqref{form20} from \eqref{form23}, we use the vertical projection $\pi := \pi_{P}$ to the subgroup $P$, introduced in Section \ref{ss:intrLip}. The most central features of $\pi$, for now, are that $\pi^{-1}\{w\} = \ell_{w}$ for $w \in P$, and that $\pi$ does not increase $3$-dimensional Hausdorff measure (too much): there exists a constant $C \geq 1$ such that
\begin{equation}\label{form28} \calH^{3}(\pi(A)) \leq C\calH^{3}(A), \qquad A \subset \He. \end{equation}
For a proof, see \cite[Lemma 3.6]{CFO}. To apply these facts, let $F \colon P \cap B(0,2) \to \He$ be the map $F(w) := p_{w,+}$. It follows from the discussion leading to \eqref{eq:dist_obs} that $F(w) = w \cdot (h_+(w),0,0)$ and hence $F$ is a Borel function. We deduce that the push-forward measure $\nu:=F_{\sharp}(\calH^{3}|_{B(0,2)\cap P})$, defined by $\nu(A):= \calH^3(B(0,2)\cap P \cap F^{-1}(A))$, is a Borel measure on $\He$ and the following integration formula holds,
\begin{equation}\label{form33} \int_{B(0,2)\cap P} \dist(p_{w,+},P) \, dw = \int_{\He} \dist(q,P) \, d \nu(q), \end{equation}
see for instance \cite[Theorem 1.19]{zbMATH01249699}.
Clearly $\nu(\He \setminus F(P\cap B(0,2))) =0$, which shows that $\mathrm{spt} \, \nu \subseteq \overline{F(P\cap B(0,2))}$. Moreover,
\begin{displaymath} \nu \ll \calH^{3}|_{\overline{F(P\cap B(0,2))}} \end{displaymath}
with bounded density, because $F^{-1}(A) \subset \pi(A)$ for all $A \subset \He$, and hence
\begin{displaymath} \nu(A) =\calH^{3}([B(0,2)\cap P] \cap F^{-1}(A)) \leq \calH^{3}(\pi(A))  \leq C\calH^{3}(A), \qquad A \subset \He, \end{displaymath}
using \eqref{form28}. Finally, we observe that
\begin{displaymath}
\overline{F(P\cap B(0,2))} \subseteq \overline{B(0,3)} \cap (P \cup \partial [\Omega \bigtriangleup H])\subseteq B(0,4) \cap (P \cup \partial \Omega).
\end{displaymath}
The last inclusion follows from the generalities $\partial [A \cup B], \partial [A \cap B] \subset \partial A \cup \partial B$:
\begin{displaymath} \partial [\Omega \bigtriangleup H] \subset \partial [\Omega \cap H^{c}] \cup \partial [\Omega^{c} \cap H] \subset \partial \Omega \cup \partial H. \end{displaymath}
In cases (a) and (b) we have $\partial H = P$, while in cases (c) and (d) the boundary of $H$ does not intersect $B(0,4)$. Combining these observations with \eqref{form33}, we find that
\begin{align*}
\int_{B(0,2)\cap P} \dist(p_{w,+},P) \, dw & \lesssim \int_{ B(0,4) \cap \partial \Omega} \dist(q,P) \, d\calH^{3}(q).
\end{align*}
Hence the right hand side of \eqref{form23} is bounded by a constant times the right hand side of \eqref{form20}. The proof of \eqref{form20}, and of Theorem \ref{oscillationBetaComparison}, is complete.

\medskip

We conclude the section by the following strengthening of Theorem \ref{oscillationBetaComparison} in the case when $\partial \Omega$ is $3$-regular:
\begin{cor}\label{oscillationAndBetas} Assume that $\Omega \subset \He$ is an open set such that $\partial \Omega$ is $3$-regular. Then,
\begin{displaymath} \frac{\w_{\Omega}(B(p,r))(s)}{r^{4}} \lesssim \beta_{\partial \Omega,1}(B(p,24r)), \qquad p \in \partial \Omega, \: 0 < s \leq r. \end{displaymath}
\end{cor}

\begin{proof} As usual, we may assume that $p = 0 \in \partial \Omega$ and $r = 1$. The proof is based on the following general observation that if $E \subset \He$ is $3$-regular, and $P \subset \He$ is a vertical plane with $P \cap B(0,2) \neq \emptyset$, then
\begin{equation}\label{form29} \dist(q,P) \lesssim \left( \int_{B(0,2)\cap E} d(x,P) \, d\calH^{3}(x) \right)^{1/4}, \qquad q \in E \cap B(0,1). \end{equation}
In Euclidean space, the analogous argument can be found for example in \cite[(5.4)]{DS1}. To prove \eqref{form29}, denote the right hand side as $\beta^{1/4}$, and assume to reach a contradiction that there exists a point $q \in B(0,1)\cap E$ with $d(q,P) \geq C\beta^{1/4}$ for some large constant $C \geq 1$. We record that this implies that $C\beta^{1/4}/4 \leq 1$, since we assumed $P \cap B(0,2) \neq \emptyset$. Also, clearly
\begin{displaymath} \dist(y,P) \geq \frac{C\beta^{1/4}}{2}, \qquad y \in E \cap B(q,C\beta^{1/4}/4) \subset B(0,2). \end{displaymath}
By $3$-regularity,
\begin{align*} (C\beta^{1/4})^{3} & \lesssim \calH^{3}(B(q,C\beta^{1/4}/4)\cap E)\\
& \leq \frac{2}{C\beta^{1/4}} \int_{B(q,C\beta^{1/4}/4)\cap E} d(x,P) \, d\calH^{3}(x) \leq \frac{2\beta^{3/4}}{C}, \end{align*}
and a contradiction is hence reached for $C \geq 1$ large enough.

From \eqref{form29} (with "$1$" and "$2$" replaced by "$12$" and "$24$"), choosing $P = z \cdot \W$ to be the best-approximating vertical plane for $\beta_{\partial \Omega,1}(B(0,24))$, we may now infer that
\begin{displaymath} \inf_{\W,z} \left[ \int_{B(0,24) \cap \partial \Omega} d(q,z \cdot \W) \, d\calH^{3}(q) + \left(\sup_{q \in B(0,12) \cap \partial \Omega} d(q,z \cdot \W) \right)^{4} \right] \lesssim \beta_{\partial \Omega,1}(B(0,24)). \end{displaymath}
In combination with Theorem \ref{oscillationBetaComparison} applied to $\epsilon(\delta):=\delta^4$, this inequality completes the proof. \end{proof}

\section{Boundedness of the Riesz transform}\label{s:Riesz}

\subsection{Definitions, and restating the main theorem} We now begin to relate the vertical oscillation coefficients to the boundedness of the $3$-dimensional Riesz transform in $\He$. For technical convenience, we replace the vectorial kernel $\nabla_{\He} G = (XG,YG)$ from the introduction with the complex kernel
\begin{displaymath} K(p) = XG(p) - iYG(p), \end{displaymath}
where $G(p) = c\|p\|_{Kor}^{-2}$ is still fundamental solution to the sub-Laplace equation $\bigtriangleup_{\He} u = 0$. For the time being, we will only need to know that $K$ is smooth outside the origin and $-3$-homogeneous with respect to the dilations $\delta_{r}$:
\begin{displaymath} K(\delta_{r}(q)) = r^{-3} K(q), \qquad q \in \He \setminus \{0\}. \end{displaymath}
It follows that $|K(q)| \lesssim \|q\|^{-3}$ for $q \in \He \setminus \{0\}$. To the kernel $K$ we associate the $\epsilon$-truncated SIOs
\begin{displaymath} \calR_{\epsilon}(\mu)(p) := \int_{\{q \in \He : \|q^{-1} \cdot p\| \geq \epsilon\}} K(q^{-1} \cdot p) \, d\mu(q), \end{displaymath}
where $\mu$ is any complex measure on $\He$ with finite total variation.

Let $\mu$ be a locally finite Borel measure on $\He$. We say that \emph{$\calR$ is bounded on $L^{2}(\mu)$}, if the operators $\mathcal{R}_{\epsilon}$ are bounded on $L^{2}(\mu)$ uniformly in $\epsilon > 0$:
\begin{displaymath} \|\calR_{\epsilon}(f\mu)\|_{L^{2}(\mu)} \leq A\|f\|_{L^{2}(\mu)}, \qquad f \in L^{1}(\mu) \cap L^{2}(\mu), \: \epsilon > 0. \end{displaymath}
The measures $\mu$ relevant here are $3$-regular measures on intrinsic Lipschitz graphs. For intrinsic Lipschitz graphs $\Gamma \subset \He$ as in Theorem \ref{mainIntro}, we will directly prove the $L^{2}(\mu)$-boundedness of $\mathcal{R}$ for the particular measure
\begin{displaymath} \mu := \calS^{3}|_{\Gamma}, \end{displaymath}
where $\calS^{3}$ is the $3$-dimensional spherical Hausdorff measure defined using the metric $d$ from \eqref{eq:metric}. This choice makes it more straightforward to use the divergence theorem, but is otherwise arbitrary. In particular, once the $L^{2}(\calS^{3}|_{\Gamma})$-boundedness of $\mathcal{R}$ has been established, then it is easy to check (or see \cite[Lemma 3.1]{CFO2}) that $\calR$ is bounded on $L^{2}(\mu)$ with respect to any $3$-regular measure $\mu$ supported on $\Gamma$ -- in particular $\calH^{3}|_{\Gamma}$.

So, here is more precisely the result we will prove below:
\begin{thm}\label{main} Let $\W \subset \He$ be a vertical subgroup, which we identify with $\{(y,t) : y,t \in \R\}$. Let $\phi \colon \W \to \R$ be an intrinsic Lipschitz function, let
\begin{displaymath} \Omega := \{(x,y,t) : x > \phi(\pi_{\W}(x,y,t))\} \end{displaymath}
be the super-graph of $\phi$, and assume that
\begin{displaymath} \int_{0}^{\infty} \osc_{\Omega}(B(p,r)) \, \frac{dr}{r} \leq C < \infty, \qquad p \in \Gamma. \end{displaymath}
Then, $\calR$ is bounded on $L^{2}(\calS^{3}|_{\Gamma_{\phi}})$.
\end{thm}
It is easy to check that $\He \setminus \Gamma_{\phi}$ has exactly two connected components, namely the super-graph $\Omega$ above, and the sub-graph $\Omega' := \{(x,y,t) : x < \phi(\pi_{\W}(x,y,t))\}$. Since
\begin{displaymath} \osc_{\Omega}(B(p,r)) = \osc_{\He \setminus \Omega}(B(p,r)) = \osc_{\Omega'}(B(p,r)), \qquad p \in \Gamma, \: r > 0, \end{displaymath}
fixing the the complementary component in Theorem \ref{main} does not render the statement less general than that of Theorem \ref{mainIntro} in the introduction.

\subsection{Test functions and the divergence theorem} We will prove Theorem \ref{main} by verifying the conditions of Christ's local $T(b)$ theorem \cite{Christ}. We first introduce some more notation. From now on the intrinsic Lipschitz graph $\Gamma := \Gamma_{\phi}$ will be fixed as in Theorem \ref{main}, and we write $\mu := \calS^{3}|_{\Gamma}$. We define the following complex-valued function $\nu$ on $\Gamma$:
\begin{equation}\label{normal} \nu(w \cdot \phi(w)) := \nu_{1}(w \cdot \phi(w)) + i\nu_{2}(w \cdot \phi(w)) := \frac{1}{\sqrt{1 + (\nabla^{\phi}\phi(w))^{2}}} + i\frac{-\nabla^{\phi}\phi(w)}{\sqrt{1 + (\nabla^{\phi}\phi(w))^{2}}}, \end{equation}
where $\nabla^{\phi}\phi$ is the intrinsic gradient of $\phi$. Since $\phi$ is intrinsic Lipschitz, $\nu(p)$ exists for $\mu$ almost every $p \in \Gamma$, because $\nabla^{\phi}\phi(w)$ exists for $\calS^{3}$ almost every $w \in \W$, and the graph map $\Phi(w) = w \cdot \phi(w)$ preserves $\calS^{3}$ null sets by the area formula for intrinsic Lipschitz functions, \cite[Theorem 1.6]{MR3168633}. By similar reasoning, $\nu \in L^{\infty}(\mu)$.

We also define the $\R^{2}$-valued map
\begin{displaymath} \nu_{H}(q) = (\nu_{1}(q),\nu_{2}(q)) = \left(\frac{1}{\sqrt{1 + (\nabla^{\phi}\phi(w))^{2}}},\frac{-\nabla^{\phi}\phi(w)}{\sqrt{1 + (\nabla^{\phi}\phi(w))^{2}}}\right) \in \R^{2}, \quad q = w \cdot \phi(w). \end{displaymath}
Then, by \cite[Corollary 4.2]{MR3168633}, $\nu_{H}$ is the inward-pointing horizontal normal of the intrinsic super-graph $\Omega = \{(x,y,t) : x > \phi(\pi_{\W}(x,y,t))\}$, expressed in the frame $\{X,Y\}$. With this notation, we have the following divergence theorem, due to Franchi, Serapioni and Serra Cassano \cite{FSSC}:
\begin{thm}[Divergence theorem]\label{divergenceTheorem} Let $V \in \mathcal{C}^{1}_{c}(\R^{3},\R^{2})$, and let $\Gamma = \Gamma_{\phi}$ be an intrinsic Lipschitz graph as above. Then,
\begin{displaymath} -\int_{\Omega} \Div V(p) \, dp = c \int_{\Gamma} \langle V,\nu_{H} \rangle \, d\mathcal{S}^{3}, \end{displaymath}
where $\Omega = \{(x,y,t) : x > \phi(\pi_{\W}(x,y,t))\}$, and $c > 0$ is a constant.
\end{thm}

\begin{remark} The divergence theorem in \cite{FSSC} looks a little different than Theorem \ref{divergenceTheorem} above, so a few remarks are in order. First, the sub- and super-graphs of intrinsic Lipschitz graphs are $\He$-Caccioppoli sets by \cite[Theorem 4.18]{FSSC2}, so \cite[Corollary 7.6]{FSSC} gives the formula
\begin{displaymath} -\int_{\Omega} \Div V(p) \, dp = c \int_{\partial_{\ast,\He}\Omega} \langle V,\nu_{H} \rangle \, d\mathcal{S}^{3}, \qquad V \in \mathcal{C}^{1}_{c}(\R^{3},\R^{2}). \end{displaymath}
Here $\partial_{\ast,\He}\Omega$ stands for the measure theoretic boundary of $\Omega$, see \cite[Definition 7.4]{FSSC}. But for domains $\Omega$ bounded by intrinsic Lipschitz graphs $\Gamma$, the measure theoretic boundary of $\Omega$ equals the topological boundary $\partial \Omega = \Gamma$: the inclusion $\Gamma \subset \partial_{\ast,\He} \Omega$ follows from basic definitions, and the inclusion $\partial_{\ast,\He} \Omega \subset \Gamma$ follows from \cite[Lemma 7.5(i)]{FSSC}.
\end{remark}

We now use the complex function $\nu$ to specify a collection of accretive test functions. Let $\psi \colon \He \to [0,1]$ be a smooth function with $\chi_{B(0,1/2)} \leq \psi \leq \chi_{B(0,1)}$, and let
\begin{displaymath} \psi_{B(p,r)}(q) := \psi(\delta_{1/r}(p^{-1} \cdot q)) \end{displaymath}
be a rescaled version of $\psi$ with $\spt \psi_{B(p,r)} \subset B(p,r)$. We record that
\begin{equation}\label{nablaPsi} |\nabla_{\He} \psi_{B(p,r)}| \lesssim \frac{\chi_{B(p,r)}}{r} \quad \text{and} \quad |\partial_{t} \psi_{B(p,r)}| \lesssim \frac{\chi_{B(p,r)}}{r^{2}}. \end{equation}
We set
\begin{displaymath} b_{B(p,r)} := \psi_{B(p,r)}\nu, \qquad p \in \Gamma, \: r > 0. \end{displaymath}
Then, recalling the formula \eqref{normal} for $\nu$, we note that
\begin{displaymath} \|b_{B(p,r)}\|_{L^{\infty}(\mu)} \lesssim 1 \quad \text{and} \quad \textrm{Re } \left( \int b_{B(p,r)} \, d\mu \right) \gtrsim \mu(B(p,r)) \end{displaymath}
for all $B(p,r)$ with $p \in \Gamma$ and $r > 0$. According to \cite[Main Theorem 10]{Christ}, the $L^{2}(\mu)$ boundedness of $\calR$ will follow once we verify the testing conditions
\begin{equation}\label{testingConditions} \|\calR_{\epsilon}(b_{B}\mu)\|_{L^{\infty}(\mu)} \leq C \quad \text{and} \quad \|\calR_{\epsilon}^{\ast}(b_{B}\mu)\|_{L^{\infty}(\mu)} \leq C \end{equation}
for all balls $B = B(p,r)$ centred on $\Gamma$, with $C \geq 1$ independent of $\epsilon > 0$. Here $\calR_{\epsilon}^{\ast}$ is the adjoint of $\calR_{\epsilon}$ with kernel
\begin{displaymath} K^{\ast}(p) = K(p^{-1}). \end{displaymath}
In fact, it will be technically more convenient to verify the testing conditions \eqref{testingConditions} for \emph{smooth truncations} of $\calR$. By a smooth truncation, we mean the operator $\calR_{s,\epsilon}$ associated to the kernel
\begin{equation}\label{Kepsilon} K_{\epsilon} := \varphi_{\epsilon}K, \end{equation}
where $\varphi$ is smooth and radially symmetric with
\begin{displaymath} \chi_{\He \setminus B(0,2)} \leq \varphi \leq \chi_{\He \setminus B(0,1)}, \end{displaymath}
and $\varphi_{\epsilon}(p) := \varphi(\delta_{1/\epsilon}(p))$ for $p \in \He$. For future reference, we remark that
\begin{equation}\label{varphiEpsilon} |\nabla_{\He}\varphi_{\epsilon}| \lesssim \frac{1}{\epsilon} \cdot \chi_{B(0,2\epsilon) \setminus B(0,\epsilon)} \quad \text{and} \quad |\partial_{t}\varphi_{\epsilon}| \lesssim \frac{1}{\epsilon^{2}} \cdot \chi_{B(0,2\epsilon) \setminus B(0,\epsilon)}.  \end{equation}
Also, if $\epsilon = 2^{-N}$ for some $N \in \N$, then $\varphi_{\epsilon}$ can be expanded as a series
\begin{equation}\label{series} \varphi_{\epsilon} = \varphi_{2^{-N}} = \sum_{j \leq N} (\varphi_{2^{-j}} - \varphi_{2^{-j + 1}}) =: \sum_{j \leq N} \eta_{j}, \end{equation}
noting that $\eta_{j}$ is supported on the annulus $B(0,2^{-j + 2}) \setminus B(0,2^{-j})$. We will assume without loss of generality that $\epsilon$ has this form in the sequel.

Now, instead of \eqref{testingConditions}, we will check that
\begin{equation}\label{testingSmooth} \|\calR_{s,\epsilon}(b_{B}\mu)\|_{L^{\infty}(\mu)} \leq C \quad \text{and} \quad \|\calR_{s,\epsilon}^{\ast}(b_{B}\mu)\|_{L^{\infty}(\mu)} \leq C \end{equation}
for all balls $B$ centred on $\Gamma$, and for some constant $C \geq 1$ independent of $\epsilon > 0$. It is easy to check that
\begin{displaymath} |\calR_{s,\epsilon}(f) - \calR_{\epsilon}(f)| \lesssim M_{\mu}(f)(p) \end{displaymath}
for all $f \in L^{\infty}(\mu)$ and $p \in \Gamma$, where $M_{\mu}$ is the Hardy-Littlewood maximal function
\begin{displaymath} M_{\mu}f(p) = \sup_{r > 0} \fint_{B(p,r)} |f(q)| \, d\mu(q), \end{displaymath}
Since $\|M_{\mu}(b_{B}\mu)\|_{L^{\infty}(\mu)} \lesssim \|b_{B}\|_{L^{\infty}(\mu)} \lesssim 1$, we see that \eqref{testingSmooth} implies \eqref{testingConditions}.

\subsection{Initial reductions for verifying the testing conditions}\label{initialReductions} We start by verifying the first condition in \eqref{testingSmooth}, that is, proving that
\begin{equation}\label{form5} |\mathcal{R}_{s,\epsilon}(b_{B}\mu)(p)| \leq C, \qquad p \in \Gamma. \end{equation}
The arguments concerning the second testing condition in \eqref{testingSmooth} will be very similar. To prove \eqref{form5}, we make a few reductions, which show that it suffices to verify \eqref{form5} for $p = 0 \in \Gamma$ and for a ball $B$ with $\dist(0,B) \leq \diam(B) = 1$

As a first step, we argue that it suffices to consider $p \in \Gamma$ with
\begin{equation}\label{distPB} \dist(p,B) \leq \diam(B). \end{equation}
Indeed, \eqref{form5} follows from standard kernel estimates if $\dist(p,B) > \diam(B)$. To see this, write $B = B(p_{0},r)$, and fix $p \in \Gamma$ with $\dist(p,p_{0}) \geq 2r$. Then $d(p,q) \geq r$ for all $q \in B$, and consequently
\begin{displaymath} |\calR_{s,\epsilon}(b_{B})(p)| \lesssim \|b_{B}\|_{L^{\infty}(\mu)} \int_{B} \frac{d\mu(q)}{d(p,q)^{3}} \lesssim \frac{\mu(B)}{r^{3}} \sim 1. \end{displaymath}
So, in the sequel we may assume that \eqref{distPB} holds.

Next, we argue that it suffices to consider the case $p = 0 \in \Gamma$. Indeed, note first that
\begin{displaymath} \tilde{\mu} := \calS^{3}|_{p^{-1} \cdot \Gamma} = (\tau_{p^{-1}})_{\sharp}\calS^{3}|_{\Gamma} = (\tau_{p^{-1}})_{\sharp}\mu. \end{displaymath}
Then, write
\begin{displaymath} \tilde{b}_{p^{-1} \cdot B} := \psi_{p^{-1} \cdot B}\nu_{p^{-1} \cdot \Gamma}, \end{displaymath}
where $\nu_{p^{-1} \cdot \Gamma}$ is the analogue of $\nu$ (recall \eqref{normal}) for the left-translated intrinsic Lipschitz graph $p^{-1} \cdot \Gamma$. In particular,
\begin{displaymath} \nu_{p^{-1} \cdot \Gamma}(p^{-1} \cdot q) = \nu(q), \qquad q \in \Gamma, \end{displaymath}
so that
\begin{displaymath} \tilde{b}_{p^{-1} \cdot B}(p^{-1} \cdot q) = \psi_{B}(q)\nu(q) = b_{B}(q), \qquad q \in \Gamma. \end{displaymath}
Using this equation, we infer that
\begin{align*} \calR_{s,\epsilon}(\tilde{b}_{p^{-1} \cdot B}\tilde{\mu})(0) & = \int_{p^{-1} \cdot \Gamma} K_{\epsilon}(q^{-1})\tilde{b}_{p^{-1} \cdot B}(q) \, d\calS^{3}(q)\\
& = \int K_{\epsilon}(q^{-1})\tilde{b}_{p^{-1} \cdot B}(q) \, d[(\tau_{p^{-1}})_{\sharp}\mu](q)\\
& = \int_{\Gamma} K_{\epsilon}((p^{-1} \cdot q)^{-1})\tilde{b}_{p^{-1} \cdot B}(p^{-1} \cdot q) \, d\calS^{3}(q)\\
& = \int_{\Gamma} K_{\epsilon}(q^{-1} \cdot p)b_{B}(q) \, d\calS^{3}(q) = \calR_{s,\epsilon}(b_{B}\mu)(p). \end{align*}
This shows that, to find a bound for $\calR_{s,\epsilon}(b_{B}\mu)(p)$, it suffices to do so for $\calR_{s,\epsilon}(\tilde{b}_{p^{-1} \cdot B}\tilde{\mu})(0)$. But the intrinsic Lipschitz graph $p^{-1} \cdot \Gamma$ has all the same properties as we assumed from $\Gamma$ in Theorem \ref{main}: the intrinsic Lipschitz constants do not change, nor do the bounds for the vertical oscillation numbers, recalling Lemma \ref{basicProp}. So, we may assume that $p = 0 \in \Gamma$.

Finally, we argue that we may assume $\diam(B) = 1$. For this purpose, we first note that
\begin{equation}\label{form11} r^{3} \cdot \delta_{r\sharp}\mu = S^{3}|_{\delta_{r}(\Gamma)} =: \tilde{\mu}. \end{equation}
Indeed, if $A \subset \delta_{r}(\Gamma)$, then $\delta_{1/r}(A) \subset \Gamma$, hence
\begin{displaymath} r^{3} \cdot (\delta_{r\sharp}\mu)(A) = r^{3} \calS^{3}(\Gamma \cap \delta_{1/r}(A)) = \calS^{3}(\delta_{r}(\Gamma) \cap A) = \tilde{\mu}(A), \end{displaymath}
which proves \eqref{form11}. Now, let $r := \diam(B)$, and let $\tilde{b}_{\delta_{1/r}(B)} := \psi_{\delta_{1/r}(B)} \cdot \nu_{\delta_{1/r}(\Gamma)}$, where $\nu_{\delta_{1/r}(\Gamma)}$ stands for the analogue of $\nu$ for the dilated intrinsic Lipschitz graph $\delta_{1/r}(\Gamma)$. In particular, it is easy to check that
\begin{displaymath} \tilde{b}_{\delta_{1/r}(B)}(\delta_{1/r}(q)) = b_{B}(q), \qquad q \in \Gamma. \end{displaymath}
We also record the equation
\begin{displaymath} K_{\epsilon}(\delta_{r}(q)) = \varphi_{\epsilon}(\delta_{r}(q))K(\delta_{r}(q)) = r^{-3} \cdot \varphi_{\epsilon/r}(q)K(q) = r^{-3}K_{\epsilon/r}(q), \end{displaymath}
using the definition of the kernel $K_{\epsilon}$ from \eqref{Kepsilon}, and the $-3$-homogeneity of $K$. Then, we may use \eqref{form11} and the equations above as follows:
\begin{align*} \calR_{s,\epsilon/r}(\tilde{b}_{\delta_{1/r}(B)}\tilde{\mu})(0) & = \int_{\delta_{1/r}(\Gamma)} K_{\epsilon/r}(q^{-1})\tilde{b}_{\delta_{1/r}(B)}(q) \, d\calS^{3}(q)\\
& = r^{-3} \int K_{\epsilon/r}(q^{-1})\tilde{b}_{\delta_{1/r}(B)}(q) \, d\delta_{(1/r)\sharp}\mu(q)\\
& = r^{-3} \int_{\Gamma} K_{\epsilon/r}([\delta_{1/r}(q)]^{-1})\tilde{b}_{\delta_{1/r}(B)}(\delta_{1/r}(q)) \, d\calS^{3}(q)\\
& = \int_{\Gamma} K_{\epsilon}(q^{-1})b_{B}(q) \, d\calS^{3}(q) = \calR_{s,\epsilon}(b_{B}\mu)(0).   \end{align*}
So, to estimate $\calR_{s,\epsilon}(b_{B}\mu)(0)$, it suffices to estimate $\calR_{s,\epsilon/r}(\tilde{b}_{\delta_{1/r}(B)}\tilde{\mu})(0)$. But, arguing as in the previous reduction, $\delta_{1/r}(\Gamma)$ is an intrinsic Lipschitz graph with the same properties as $\Gamma$. So, in the sequel we assume that $\diam(B) = 1$.

Summarising, we have reduced the proof of \eqref{form5} to the case
\begin{equation}\label{allReductions} p = 0 \in \Gamma \quad \text{and} \quad \dist(0,B) \leq \diam(B) = 1. \end{equation}

\subsection{Verifying the testing conditions} With the above reductions in mind, we start the proof of \eqref{form5}. We record that
\begin{equation}\label{form10} K(q^{-1}) = -\widetilde{X}G(q) + i\widetilde{Y}G(q), \qquad q \in \He \setminus \{0\}, \end{equation}
as a straightforward computation shows. Hence, we may write
\begin{align*} \calR_{s,\epsilon}(b_{B}\mu)(0) & = \int_{\Gamma} \varphi_{\epsilon}(q)(-\widetilde{X}G(q) + i\widetilde{Y}G(q))b_{B}(q) \, d\calS^{3}(q)\\
& = -\int_{\Gamma} \langle \psi_{B}(q)\varphi_{\epsilon}(q)\widetilde{\nabla}_{\He}G( q),\nu_{H}(q) \rangle \, d\calS^{3}(q)\\
&\quad + i\int_{\Gamma} \langle \psi_{B}(q)\varphi_{\epsilon}( q)[\widetilde{Y}G(q),-\widetilde{X}G(q)],\nu_{H}(q) \rangle \, d\calS^{3}(q) =: I_{1} + iI_{2}, \end{align*}
recalling the notation from Section \ref{s:vector_fields}. In order to evaluate $I_1$ and $I_2$, respectively, we will apply the divergence theorem (Theorem \ref{divergenceTheorem}) to the vector fields $$V_{1} :=(\psi_B  \varphi_{\epsilon}\widetilde{X} G , \psi_B \varphi_{\epsilon} \widetilde{Y} G) \in \mathcal{C}^{\infty}_{c}(\R^{3},\R^{2})$$ and $$V_{2} :=(\psi_B  \varphi_{\epsilon}\widetilde{Y} G , - \psi_B \varphi_{\epsilon} \widetilde{X} G) \in \mathcal{C}^{\infty}_{c}(\R^{3},\R^{2}),$$ respectively.
\subsubsection{Estimate for $I_{1}$} After an application of Theorem \ref{divergenceTheorem}, $I_{1}$ becomes
\begin{align*} I_{1} & = -c\int \Div (\psi_{B}(q)\varphi_{\epsilon}(q) \widetilde{\nabla}_{\He}G(q)) \, dq\\
& = - c\int_{\Omega} \langle \nabla_{\He} (\psi_{B}{\varphi}_{\epsilon})(q),\widetilde{\nabla}_{\He}G( q) \rangle \, dq - c\int_{\Omega} (\psi_{B}\varphi_{\epsilon})(q)\Div \widetilde{\nabla}_{\He}G(q) \, dq =: -cI_{1}^{1} - cI_{1}^{2}. \end{align*}
For $I_{1}^{1}$, we infer from \eqref{nablaPsi}, \eqref{varphiEpsilon}, and the product rule that
\begin{displaymath} |\nabla_{\He}(\psi_{B}\varphi_{\epsilon})| \lesssim \frac{1}{\epsilon} \cdot \chi_{B(0,2\epsilon) \setminus B(0,\epsilon)} + \chi_{B}. \end{displaymath}
Since moreover $|\widetilde{\nabla}_{\He}G(q)| \lesssim \| q\|^{-3}$ (this follows from \eqref{form10} for instance), we get
\begin{equation}\label{form7}
\left|\int_{\Omega} \langle \nabla_{\He} (\psi_{B}\varphi_{\epsilon})(q),\widetilde{\nabla}_{\He}G( q) \rangle \, dq \right|
\lesssim \frac{1}{\epsilon} \int_{B(0,2\epsilon) \setminus B(0,\epsilon)} \|q\|^{-3} dq + \int_{B} \|q\|^{-3} \, dq \lesssim 1.
\end{equation}

To handle the term $I_{1}^{2}$, we observe the following general relationship between left and right divergence:
\begin{equation}\label{leftVsRight} \Div (V_{1},V_{2}) = \widetilde{\Div} (V_{1},V_{2}) + \partial_{t}(-y V_{1} + xV_{2}), \qquad (V_{1},V_{2}) \in \mathcal{C}^{1}(\R^{3},\R^{2}). \end{equation}
It follows that
\begin{displaymath} I_{1}^{2} = \int_{\Omega} (\psi_{B}\varphi_{\epsilon})(q) \widetilde{\Div} \widetilde{\nabla}_{\He}G(q) \, dq + \int_{\Omega} (\psi_{B}\varphi_{\epsilon})(q)\partial_{t}(-y\tilde{X}G(q) + x\tilde{Y}G(q)) \, dq. \end{displaymath}
Here
\begin{displaymath} \widetilde{\Div}\widetilde{\nabla}_{\He}G(q) = \widetilde{\bigtriangleup}_{\He} G(q) = 0, \qquad q \in \spt \varphi_{\epsilon}, \end{displaymath}
since $G$ is simultaneously the fundamental solution for both operators $\bigtriangleup_{\He}$ and $\widetilde{\bigtriangleup}_{\He}$. So, the first term vanishes. Consequently,
\begin{equation}\label{form6} I_{1}^{2} =: \int_{\Omega} (\psi_{B}\varphi_{\epsilon})(q)\partial_{t} \tilde{K}(q) \, dq = \int_{\Omega} \partial_{t}(\psi_{B}\varphi_{\epsilon}\tilde{K})(q) \, dq - \int_{\Omega} \partial_{t}(\psi_{B}\varphi_{\epsilon})(q)\tilde{K}(q) \, dq, \end{equation}
where $\tilde{K}$ is the $-2$-homogeneous kernel
\begin{displaymath} \tilde{K}(z,t) = -y\tilde{X}G(z,t) + x\tilde{Y}G(z,t) = \frac{8t|z|^{2}}{\|(z,t)\|_{Kor}^{6}}, \qquad z = (x,y). \end{displaymath}
The main term in \eqref{form6} is the first one, because the second one can be treated in the same fashion as $I_{1}^{1}$ above. Indeed, simply notice from \eqref{nablaPsi}, \eqref{varphiEpsilon}, and the product rule that
\begin{displaymath} |\partial_{t}(\psi_{B}\varphi_{\epsilon})(q)| \lesssim \frac{1}{\epsilon^{2}} \chi_{B(0,2\epsilon) \setminus B(0,\epsilon)} + \chi_{B}, \end{displaymath}
so that
\begin{align*} \left| \int_{\Omega} \partial_{t}(\psi_{B}\varphi_{\epsilon})(q)\tilde{K}(q) \, dq \right| & \lesssim \frac{1}{\epsilon^{2}} \int_{B(0,2\epsilon) \setminus B(0,\epsilon)} |\tilde{K}(q)| \, dq + \int_{B} |\tilde{K}(q)| \, dq\\
& \lesssim \frac{1}{\epsilon^{4}} \calH^{4}(B(0,2\epsilon)) + 1 \sim 1. \end{align*}
Finally, the first term in \eqref{form6} is handled using \eqref{series} and Lemma \ref{mainLemma} (noting that $\spt (\psi_{B}\eta_{j}\tilde{K}) \subset B(0,s)$ for any $s \in [2^{-j + 2},2^{-j + 3}]$):
\begin{align*} \left| \int_{\Omega} \partial_{t}(\psi_{B}\varphi_{\epsilon}\tilde{K})(q) \, dq \right| & \leq \sum_{j \leq N} \left|  \int_{\Omega} \partial_{t}(\psi_{B}\eta_{j}\tilde{K})(q) \, dq \right|\\
& \lesssim \sum_{j \leq N} 2^{-4j} \|\partial_{t}(\psi_{B}\eta_{j}\tilde{K})\|_{\infty} \int_{2^{-j + 2}}^{2^{-j + 3}} \osc_{\Omega}(B(0,10s)) \, \frac{ds}{s} \end{align*}
From the product rule, noting that
\begin{itemize}
\item $\spt \eta_{j} \subset B(0,2^{-j + 2}) \setminus B(0,2^{-j})$,
\item $\spt \psi_{B} \subset B \subset B(0,2)$ by \eqref{allReductions},
\item $\tilde{K}$ is $-2$-homogeneous, and
\item $\partial_{t}\tilde{K}$ is $-4$-homogeneous,
\end{itemize}
we see that
\begin{displaymath} \|\partial_{t}(\psi_{B}\eta_{j}\tilde{K})\|_{\infty} \lesssim \begin{cases} 2^{4j}, & j \geq -1\\  0, & j < -1. \end{cases} \end{displaymath}
To verify the last bullet point, one can simply compute that $\partial_{t} \tilde{K}$ is the kernel
\begin{displaymath} \partial_{t} \tilde{K}(z,t) = 8\frac{|z|^{2}(|z|^{4} - 32t^{2})}{\|(z,t)\|_{Kor}^{10}}, \qquad z = (x,y). \end{displaymath}
Summarising the estimate above, we have now shown that
\begin{displaymath} |I_{1}| \lesssim 1 + \sum_{-1 \leq j \leq N} \int_{2^{-j + 2}}^{2^{-j + 3}} \osc_{\Omega}(B(0,10s)) \, \frac{ds}{s} \lesssim 1 + \int_{0}^{\infty} \osc_{\Omega}(B(0,s)) \, \frac{ds}{s} \leq 1 + C. \end{displaymath}

\subsubsection{Estimate for $I_{2}$} We move to the term
\begin{align*} I_{2} & = \int_{\Gamma} \langle \psi_{B}(q)\varphi_{\epsilon}( q)[\widetilde{Y}G(q),-\widetilde{X}G(q)],\nu_{H}(q) \rangle \, d\calS^{3}(q)\\
& = - c\int_{\Omega} \Div (\psi_{B}\varphi_{\epsilon}[\widetilde{Y}G,-\widetilde{X}G])(q) \, dq\\
& = -c\int_{\Omega} \langle \nabla_{\He} (\psi_{B}\varphi_{\epsilon})(q), (\widetilde{Y}G(q),-\widetilde{X}G(q)) \rangle \, dq - c\int_{\Omega} (\psi_{B}\varphi_{\epsilon})(q) \Div [\widetilde{Y}G,-\widetilde{X}G](q) \, dq\\& =: -cI_{2}^{1} - cI_{2}^{2}. \end{align*}
where the divergence theorem was applied. The term $I_{2}^{1}$ can be handled precisely as $I_{1}^{1}$ above, see \eqref{form7}. So, we concentrate on the term $I_{2}^{2}$. Once again, due to the presence of the right-invariant vector fields $\widetilde{X}$ and $\widetilde{Y}$, it is useful to consider the right divergence instead of the left one. Recalling \eqref{leftVsRight}, and setting $p = (x,y,t)$, we write
\begin{align*} \Div [\widetilde{Y}G,-\widetilde{X}G](p) & = \widetilde{\Div} [\widetilde{Y}G,-\widetilde{X}G](p) + \partial_{t}(-y\widetilde{Y}G - x\widetilde{X}G)(p)\\
& = (\widetilde{X}\widetilde{Y}G - \widetilde{Y}\widetilde{X}G)(p) + \partial_{t} \widehat{K}(p)\\
& = -\partial_{t}G(p) + \partial_{t}\widehat{K}(p). \end{align*}
Here $\widehat{K}$ is yet another $-2$-homogeneous kernel with explicit expression
\begin{displaymath} \widehat{K}(z,t) = \frac{2|z|^{4}}{\|(z,t)\|_{Kor}^{6}}, \qquad (z,t) \in \He \setminus \{0\}. \end{displaymath}
In other words,
\begin{equation}\label{I22} I_{2}^{2} = -\int_{\Omega} (\psi_{B}\varphi_{\epsilon})(q)\partial_{t}G(q) \, dq + \int_{\Omega} (\psi_{B}\varphi_{\epsilon})(q)\partial_{t}\widehat{K}(q) \, dq. \end{equation}
From this point on, the treatment of both terms can be continued as on line \eqref{form6} above. The only facts we needed about the kernel $\widetilde{K}$ there was that it is $-2$-homogeneous, and its $t$-derivative is $-4$-homogeneous. These properties are also satisfied for $G$ and $\widehat{K}$. In fact, the $t$-derivatives are given by
\begin{displaymath} \partial_{t} G(z,t) = \frac{16 t}{\|(z,t)\|_{Kor}^{6}} \quad \text{and} \quad \partial_{t} \widehat{K}(z,t) = -\frac{96|z|^{4}t}{\|(z,t)\|_{Kor}^{10}}. \end{displaymath}
So, continuing as in \eqref{form6}, and afterwards, we obtain
\begin{displaymath} |I_{2}^{2}| \lesssim 1 + \int_{0}^{\infty} \osc_{\Omega}(B(0,s)) \, \frac{ds}{s} \leq 1 + C. \end{displaymath}
This concludes the proof of \eqref{form5}: we have shown that
\begin{equation}\label{form8} \|\mathcal{R}_{s,\epsilon}(b_{B}\mu)\|_{L^{\infty}(\mu)} \leq C. \end{equation}
\subsubsection{The adjoint} To prove Theorem \ref{main}, it remains to establish the bound analogous to \eqref{form8} for the adjoint $\mathcal{R}^{\ast}_{s,\epsilon}$. Arguing as in Section \ref{initialReductions}, we may assume that the conditions in \eqref{allReductions} are in force. In other words, it suffices to show that
\begin{displaymath} |\calR^{\ast}_{s,\epsilon}(b_{B}\mu)(0)| \leq C, \end{displaymath}
where $B \subset \He$ is a ball with $\dist(0,B) \leq 1 = \diam(B)$, and $0 \in \Gamma$. By definition,
\begin{align*} \calR^{\ast}_{s,\epsilon}(b_{B}\mu)(0) & = \int_{\Gamma} \varphi_{\epsilon}(q)(XG(q) - iYG(q))b_{B}(q) \, d\calS^{3}(q)\\
& = \int_{\Gamma} \langle (\psi_{B}\varphi_{\epsilon})(q)\nabla_{\He}G(q),\nu_{H}(q) \, \rangle \, d\calS^{3}\\
& \qquad + i\int_{\Gamma} \langle (\psi_{B}\varphi_{\epsilon})(q)[-YG,XG](q),\nu_{H}(q) \rangle \, d\calS^{3}(q) =: J_{1} + iJ_{2}. \end{align*}
The situation is now similar to, but slightly simpler than, the one we have already treated. After we apply the divergence theorem and use the product rule, $J_{1}$ becomes
\begin{displaymath} J_{1} = -c \int_{\Omega} \langle \nabla_{\He} (\psi_{B}\varphi_{\epsilon})(q),\nabla_{\He}G(q) \rangle \, dq - c\int_{\Omega} (\psi_{B}\varphi_{\epsilon})(q) \Div \nabla_{\He}G(q) \, dq. \end{displaymath}
The second term vanishes, as $\Div \nabla_{\He} G(q) = \bigtriangleup_{\He} G(q) = 0$ for $q \in \spt \varphi_{\epsilon}$. The first term can be estimated as in \eqref{form7}.

Concerning $J_{2}$, the divergence theorem gives
\begin{displaymath} J_{2} = -c \int_{\Omega} \langle \nabla_{\He} (\psi_{B}\varphi_{\epsilon})(q),[-YG,XG](q) \rangle \, dq - c\int_{\Omega} (\psi_{B}\varphi_{\epsilon})(q) \Div [-YG,XG](q) \, dq. \end{displaymath}
Once more, the first term is estimated using the argument from \eqref{form7}. In the second term, we find that
\begin{displaymath} \Div [-YG,XG](q) = -XYG(q) + YXG(q) = -\partial_{t}G(q), \qquad q \in \He \setminus \{0\}. \end{displaymath}
From this point on, the estimates are the same as for the term $I_{2}^{2}$ above, see \eqref{I22}. We have now established that
\begin{displaymath} \|\calR^{\ast}(b_{B}\mu)\|_{L^{\infty}(\mu)} \leq C, \end{displaymath}
and the proof of Theorem \ref{main} is complete.

\section{Application: intrinsic Lipschitz graphs with extra vertical regularity}

In this section, prove Theorem \ref{main2Intro}, which we restate below:
\begin{thm}\label{main2} Let $\phi \colon \W \to \R$ be an intrinsic Lipschitz function which satisfies the following H\"older regularity in the vertical direction:
\begin{equation}\label{holder} |\phi(y,t) - \phi(y,s)| \leq H|t - s|^{(1 + \tau)/2}, \qquad |s - t| \leq 1, \end{equation}
and
\begin{equation}\label{holder2} |\phi(y,t) - \phi(y,s)| \leq H|t - s|^{(1 - \tau)/2}, \qquad |s - t| > 1. \end{equation}
where $H \geq 1$ and $0 < \tau \leq 1$. Then $R$ is bounded on $L^{2}(\calH^{3}|_{\Gamma_{\phi}})$.
\end{thm}

As a corollary, we recover the main theorem of \cite{CFO2} for the Riesz transform:

\begin{cor} Let $\W \subset \He$ be a vertical plane, let $\alpha > 0$, and let $\phi \colon \W \to \V$ be a compactly supported $C^{1,\alpha}(\W)$ in the sense of \cite{CFO2}. Then $R$ is bounded on $L^{2}(\calH^{3}|_{\Gamma_{\phi}})$.
\end{cor}

\begin{proof} By \cite[Proposition 4.2]{CFO2}, an intrinsic $C^{1,\alpha}$-function $\phi$ satisfies \eqref{holder} with exponent $\tau = \alpha$. Since $\phi$ is continuous and compactly supported, \eqref{holder2} is also satisfied if the constant $H$ is chosen large enough. To apply Theorem \ref{main2}, we still need to argue that $\phi$ is intrinsic Lipschitz: this is the content of \cite[Remark 2.18]{CFO2}. \end{proof}

Besides the compact support assumption, a notable difference between Theorem \ref{main2} and the main theorem of \cite{CFO2} is that the intrinsic $C^{1,\alpha}$-condition implies extra regularity in both vertical and horizontal directions. The conditions \eqref{holder}-\eqref{holder2}, on the other hand, imply nothing about the horizontal behaviour of $\phi$. To emphasise this, we give another corollary of Theorem \ref{main2}:
\begin{cor} Let $\phi_{0} \colon \R \to \R$ be a (Euclidean) Lipschitz function, and let $\phi(0,y,t) := \phi_{0}(y)$. Then $\mathcal{R}$ is bounded on $L^{2}(\mu)$, where $\mu$ is $\calH^{3}$ restricted to $\Gamma_{\phi}$.
\end{cor}

\begin{proof} We first note that $\phi$ is intrinsic Lipschitz, because
\begin{displaymath} |\phi(0,y,t) - \phi(0,y',t')| \lesssim |y - y'| \leq \|\pi_{\W}(\Phi(0,y',t')^{-1} \cdot \Phi(0,y,t))\|, \end{displaymath}
where $\Phi(0,y,t) = (0,y,t) \cdot (\phi(0,y,t),0,0)$ is the graph map parametrising $\Gamma_{\phi}$. Conditions \eqref{holder}-\eqref{holder2} are trivially satisfied, so the claim follows from Theorem \ref{main2}. \end{proof}

\subsection{Proof of Theorem \ref{main2}} The proof is based on the following lemma:
\begin{lemma}\label{holderOscillation} Assume that $\phi \colon \W := \{(0,y,t) : y,t \in \R\} \to \R$ is intrinsic Lipschitz and satisfies \eqref{holder}-\eqref{holder2}. Then,
\begin{equation}\label{form9} \osc_{\Omega}(B(p,r)) \lesssim H^{4}\min\{r^{\tau},r^{-\tau}\}, \qquad p \in \Gamma_{\phi}, \: 0 < r < \infty, \end{equation}
where $\Omega = \{(x,y,t) : x > \phi(\pi_{\W}(x,y,t))\}$, and the implicit constants depend on the intrinsic Lipschitz constants of $\phi$.
\end{lemma}

By Theorem \ref{main}, the lemma above will prove Theorem \ref{main2}.

\begin{proof}[Proof of Lemma \ref{holderOscillation}] The plan is to first use \eqref{holder} to establish the bound
\begin{equation}\label{form13} \osc_{\Omega}(B(p,r)) \lesssim H^{4}r^{\tau}, \qquad p \in \Gamma_{\phi}, \: 0 < r \leq 1. \end{equation}
The second bound in \eqref{form9} will follow by a similar argument from \eqref{holder2} for $r > 1$.

Write $\Gamma := \Gamma_{\phi}$, and fix $0 < r \leq 1$ and $0 < s \leq r$. We claim that
\begin{equation}\label{form12} \w_{\Omega}(B(p,r))(s) = \int_{B(p,r) \cap \Gamma(Hr^{1 + \tau})} {|\chi_{\Omega}(q) - \chi_{\Omega}(q \cdot (0,0,s^{2}))|} \, dq, \end{equation}
where $\Gamma(Hr^{1 + \tau})$ denotes the $(Hr^{1 + \tau})$-neighbourhood of $\Gamma$.
To prove this, it suffices to show that if $q \in B(p,r)$ with $\dist(q,\Gamma) > Hr^{1 + \tau}$, then
\begin{displaymath} \chi_{\Omega}(q) = \chi_{\Omega}(q \cdot (0,0,s^{2})). \end{displaymath}
Indeed, assume to the contrary that $q = (x,y,t) \in B(p,r)$ can be found with $\dist(q,\Gamma) > Hr^{1 + \tau}$ and $\chi_{\Omega}(q) \neq \chi_{\Omega}(q \cdot (0,0,s^{2}))$. This has two consequences: first, in particular
\begin{align*} |x - \phi(\pi_{\W}(x,y,t))| & = d((x,0,0),\phi(\pi_{\W}(q)))\\
& = d(\pi_{\W}(q) \cdot (x,0,0), \pi_{\W}(q) \cdot \phi(\pi_{\W}(q)))\\
& = d(q,\Phi(\pi_{\W}(q))) > Hr^{1 + \tau}, \end{align*}
where $\Phi(w) = w \cdot \phi(w)$ is the graph map parametrising $\Gamma$. Second, there exists $0 \leq u \leq s$ such that $(x,y,t + u^{2}) = q \cdot (0,0,u^{2}) \in \Gamma$, so in particular
\begin{displaymath} x = \phi(\pi_{\W}(q \cdot (0,0,u^{2}))). \end{displaymath}
Combining the information above,
\begin{displaymath} |\phi(\pi_{\W}(x,y,t + u^{2})) - \phi(\pi_{\W}(x,y,t))| > Hr^{1 + \tau}. \end{displaymath}
Spelling out the definition of $\pi_{\W}$, this is equivalent to
\begin{displaymath} Hr^{1 + \tau} < |\phi(0,y,t + u^{2} + \tfrac{1}{2}xy) - \phi(0,y,t + \tfrac{1}{2}xy)| \leq Hu^{1 + \tau} \leq Hs^{1 + \tau} \leq Hr^{1 + \tau}. \end{displaymath}
We have reached a contradiction, and hence proved \eqref{form12}.

It follows from \eqref{form12} that
\begin{displaymath} \osc_{\Omega}(B(p,r)) = \fint_{0}^{r} \frac{\w_{\Omega}(B(p,r))(s)}{r^4} \, {ds} \lesssim \frac{\calH^{4}(B(p,r) \cap \Gamma(Hr^{1 + \tau}))}{r^{4}}. \end{displaymath}
To conclude the proof, we find a maximal $Hr^{1 + \tau}$-separated set $S \subset B(p,2Hr) \cap \Gamma$; note that this step uses the assumption $r \leq 1$, so that $r^{1 + \tau} \leq r$.  Since $\Gamma$ is $3$-regular, we have
\begin{equation}\label{form14} \card S \lesssim r^{-3\tau}. \end{equation}
On the other hand, the balls $B(q,10Hr^{1 + \tau})$, $q \in S$, cover $B(p,r) \cap \Gamma(Hr^{1 + \tau})$, whence
\begin{displaymath} \osc_{\Omega}(B(p,r)) \lesssim \frac{\calH^{4}(B(p,r) \cap \Gamma(Hr^{1 + \tau}))}{r^{4}} \lesssim (\card S) \cdot \frac{(Hr^{1 + \tau})^{4}}{r^{4}} \lesssim H^{4}r^{\tau}.  \end{displaymath}
This proves \eqref{form13}.

To prove the second bound in \eqref{form9}, one fixes $r \geq 1$ and proceeds as above, using \eqref{holder2} instead of \eqref{holder}. One first obtains
\begin{displaymath} \w_{\Omega}(B(p,r))(s) = \int_{B(p,r) \cap \Gamma(Hr^{1 - \tau})} {|\chi_{\Omega}(q) - \chi_{\Omega}(q \cdot (0,0,s^{2}))|} \, dq \end{displaymath}
This leads to $\osc_{\Omega}(B(p,r)) \lesssim \calH^{4}(B(p,r) \cap \Gamma(Hr^{1 - \tau}))/r^{4}$. Since $r \geq 1$, one has $r^{1 - \tau} \leq r$. One finally chooses a maximal $Hr^{1 - \tau}$-separated set $S \subset B(p,2Hr) \cap \Gamma$, and finds that \eqref{form14} gets replaced by $\card S \lesssim r^{3\tau}$. This gives $\osc_{\Omega}(B(p,r)) \lesssim H^{4}r^{-\tau}$, as desired.  \end{proof}

\section{Problems and remarks}\label{problems}

\subsection{Carleson packing conditions for the vertical oscillation coefficients?} Theorem \ref{mainIntro} guarantees the $L^{2}$-boundedness of $R$ on intrinsic Lipchitz graphs $\Gamma = \partial \Omega \subset \He$ satisfying the uniform condition
\begin{equation}\label{form15} \int_{0}^{\infty} \osc_{\Omega}(B(p,r)) \, \frac{dr}{r} \lesssim 1, \qquad p \in \Gamma. \end{equation}
A comparison with analogous results in Euclidean space, in particular those in \cite{DS1}, suggests that it might be possible to relax \eqref{form15} to a \emph{Carleson packing condition} for the vertical oscillation coefficients, such as the one below:
\begin{equation}\label{carleson} \int_{B(p_{0},R)}\int_{0}^{R} \osc_{\Omega}(B(p,r))^{\eta} \, d\calH^{3}(p) \, \frac{dr}{r} \lesssim R^{3}, \qquad p_{0} \in \Gamma, \: 0 < R \leq \diam \Omega. \tag{Car$(\eta)$} \end{equation}
Here $\eta \geq 1$ is a parameter, and evidently the condition \eqref{carleson} gets weaker as $\eta$ increases. Two questions now arise:
\begin{question}\label{q1} For which parameters $\eta \geq 1$ -- if any -- does the following hold? Assume that $\Gamma = \partial \Omega \subset \He$ is an intrinsic Lipschitz graph satisfying \eqref{carleson}. Then $R$ is bounded on $L^{2}(\calH^{3}|_{\Gamma})$.
\end{question}

\begin{question}\label{q2} For which parameters $\eta \geq 1$ -- if any -- does the following hold? Every intrinsic Lipschitz graph $\Gamma \subset \He$ satisfies \eqref{carleson}. \end{question}

We have no further insight on either of the questions at the moment. We conjecture that every intrinsic Lipschitz graph $\Gamma \subset \He$ satisfies \eqref{carleson} for $\eta \geq 4$.

\subsection{A connection between vertical perimeter and $\beta$-numbers}\label{perimeterAndBetas}

Let $\Omega \subset \He$ be an open set with $3$-regular boundary, and let $1 \leq p < \infty$. Recall from Remark \ref{vPerRemark} that the $L^{p}$-vertical perimeter of $\Omega$ in a ball $B(q,r)$, $q \in \partial \Omega$, is the quantity
\begin{displaymath} \wp_{\Omega,p}(B(q,r)) := \left( \int_{0}^{\infty} \left( \frac{\w_{\Omega}(B(q,r))(s)}{s} \right)^{p} \, \frac{ds}{s} \right)^{1/p}. \end{displaymath}
Given Corollary \ref{oscillationAndBetas}, it is reasonable to expect an inequality between $\wp_{\Omega,p}$ and some quantity defined via the vertical $\beta$-numbers $\beta_{\partial \Omega,1}$. Such an inequality is given by the following proposition:
\begin{proposition}\label{verticalPerimeterVsBetas} Let $\Omega \subset \He$ be a non-empty open set with $3$-regular boundary. Let $p_{0} \in \partial \Omega$ and $0 < R \leq \diam \Omega$. Then,
\begin{displaymath} \wp_{\Omega,p}(B(p_{0},R)) \lesssim R^{3} + \int_{B(p_{0},CR) \cap \partial \Omega} \left( \int_{0}^{R} \beta_{\partial \Omega,1}(B(q,Cr))^{p} \, \frac{dr}{r} \right)^{1/p} \, d\calH^{3}(q), \end{displaymath}
where $C \geq 1$ is an absolute constant.
\end{proposition}

\begin{proof} Fix $0 < r \leq R$. We start by arguing that
\begin{equation}\label{form30} \frac{\w_{\Omega}(B(p_{0},R))(r)}{r} \lesssim \int_{B(p_{0},CR) \cap \partial \Omega} \beta_{\partial \Omega,1}(B(p,Cr)) \, d\calH^{3}(p). \end{equation}
To this end, let $\calB_{r}$ be a finite family of balls of radius $r$ covering $B(p_{0},R)$ such that the concentric balls of radius $r/2$ are disjoint. Note that if $\dist(B,\partial \Omega) > 2r$, then
\begin{displaymath} |\chi_{\Omega}(q) - \chi_{\Omega}(q \cdot (0,0,r^{2}))| = 0, \qquad q \in B, \end{displaymath}
because $d(q,q \cdot (0,0,r^{2})) = 2r$ with our choice of metric $d$, recall \eqref{eq:metric}. Whenever $B \in \calB_{r}$ with $\dist(B,\partial \Omega) \leq 2r$, we pick some ball $\hat{B} \supset B$, which is centred on $\partial \Omega$ and has radius at most $5r$. By the $3$-regularity of the boundary, we then have
\begin{displaymath} \calH^{3}(\hat{B} \cap \partial \Omega) \sim r^{3}, \qquad B \in \calB_{r}, \: \dist(B,\partial \Omega) \leq 2r. \end{displaymath}
Then, by the bounded overlap of the balls $\hat{B}$, and applying Corollary \ref{oscillationAndBetas}, we can estimate as follows:
\begin{align*} \frac{\w_{\Omega}(B(p_{0},R))(r)}{r} & = \int_{B(p_{0},R)} \frac{|\chi_{\Omega}(q) -\chi_{\Omega}(q \cdot (0,0,r^{2}))|}{r} \, dq \\
& \leq \mathop{\sum_{B \in \calB_{r}}}_{\dist(B,\partial \Omega) \leq 2r} \int_{B} \frac{|\chi_{\Omega}(q) - \chi_{\Omega}(q \cdot (0,0,r^{2}))|}{r} \, dq \\
& \lesssim \mathop{\sum_{B \in \calB_{r}}}_{\dist(B,\partial \Omega) \leq 2r} \frac{\w_{\Omega}(\hat{B})(r)}{r^{4}} \calH^{3}(\hat{B} \cap \partial \Omega)\\
& \lesssim \mathop{\sum_{B \in \calB_{r}}}_{\dist(B,\partial \Omega) \leq 2r} \beta_{\partial \Omega,1}(24\hat{B})\calH^{3}(\hat{B} \cap \partial \Omega)\\
& \lesssim \int_{B(p_{0},CR)} \beta_{\partial \Omega,1}(B(q,Cr)) \, d\calH^{3}(q),  \end{align*}
This is \eqref{form30}. Applying Minkowski's integral inequality, we infer the following bound:
\begin{align*} \left( \int_{0}^{R} \left( \frac{\w_{\Omega}(B(p_{0},R))(r)}{r} \right)^{p} \, \frac{dr}{r} \right)^{1/p} & \lesssim \left( \int_{0}^{R} \left( \int_{B(p_{0},CR) \cap \partial \Omega} \beta_{\partial \Omega,1}(B(q,Cr)) \, d\calH^{3}(q) \right)^{p} \, \frac{dr}{r} \right)^{1/p}\\
& \leq \int_{B(p_{0},CR) \cap \partial \Omega} \left(\int_{0}^{R} \beta_{\partial \Omega,1}(B(q,Cr))^{p} \, \frac{dr}{r} \right)^{1/p} \, d\calH^{3}(q). \end{align*}
Finally, it remains to note that
\begin{displaymath} \left(\int_{R}^{\infty} \left(\frac{\w_{\Omega}(B(p_{0},R))(r)}{r} \right)^{p} \, \frac{dr}{r} \right)^{1/p} \lesssim \left(\int_{R}^{\infty} \frac{R^{4p}}{r^{p + 1}} \, dr \right)^{1/p} \sim R^{3}, \end{displaymath}
and the proposition follows by combining the two estimates above. \end{proof}

As an immediate corollary, we infer that if the $\beta_{\partial \Omega,1}$-numbers satisfy a Carleson packing condition similar to \eqref{carleson}, namely
\begin{equation}\label{carlesonBeta} \int_{B(p_{0},R)} \int_{0}^{R} \beta_{\partial \Omega,1}(B(q,r))^{p} \, d\calH^{3}(q) \, \frac{dr}{r} \lesssim R^{3}, \qquad p_{0} \in \partial \Omega, \: 0 < R \leq \diam \Omega, \end{equation}
then the $L^{p}$-vertical perimeter is bounded by (a constant times) the horizontal perimeter:
\begin{cor}\label{verticalPerimeterCorollary} Let $1\leq p<\infty$. Assume that $\Omega \subset \He$ is a non-empty open set with $3$-regular boundary, and assume that \eqref{carlesonBeta} holds. Then
\begin{displaymath} \wp_{\Omega,p}(B(q,r)) \lesssim r^{3}, \qquad q \in \partial \Omega, \: 0 < r \leq \diam \Omega. \end{displaymath}
\end{cor}
\begin{proof} Apply Proposition \ref{verticalPerimeterVsBetas}, then H\"older's inequality, and finally \eqref{carlesonBeta}. \end{proof}

\bibliographystyle{plain}
\bibliography{references}

\end{document}